\documentclass[11pt,reqno, a4wide]{amsart}
\usepackage{amssymb}
\usepackage{amsmath}
\usepackage{amsthm}
\usepackage{mathptmx}
\usepackage{color}
\usepackage[pdftex]{graphicx}
\usepackage{hyperref}
\usepackage{a4wide}

\usepackage[latin1]{inputenc}

\newtheorem{teorema}{Theorem}
\newtheorem{corolario}[teorema]{Corollary}
\newtheorem{defini}[teorema]{Definition}
\newtheorem{lema}[teorema]{Lemma}

\newtheorem{prop}[teorema]{Proposition}
\newtheorem{obs}[teorema]{Remark}

\newcommand{\whitebox}{$\Box$}
\newenvironment{prova}[1][Proof]{\textbf{#1.} }{\hfill\whitebox\bigskip}
\newenvironment{dem}[1][Proof]{\textbf{#1.} }{\hfill\whitebox\bigskip}

\numberwithin{equation}{section}
\newcommand{\R}{\mathbb{R}}
\newcommand{\T}{\mathbb{T}}

\newcommand{\supp}{\mathrm{supp}}

\newcommand{\esssup}{\mathrm{esssup}}

\begin{document}

\title[On the Schrödinger-Debye System]{On the Schrödinger-Debye System in Compact Riemannian Manifolds}
\author{ Marcelo Nogueira}
\address{Department of Mathematics, State University of Campinas, 13083-859, Campinas, SP, Brazil}
\email{marcelonogueira19@gmail.com }
\thanks{M. Nogueira was supported by CNPq, Brazil.}
\author{Mahendra Panthee}
\address{Department of Mathematics, State University of Campinas, 13083-859, Campinas, SP,  Brazil}
\email{mpanthee@ime.unicamp.br}
\thanks{M. Panthee was partially supported by CNPq (308131/2017-7) and FAPESP (2016/25864-6) Brazil.}


\begin{abstract}
We consider the initial value problem (IVP) associated to the Schr\"odinger-Debye system posed on a $d$-dimensional compact Riemannian manifold $M$ and  prove local well-posedness result for given data $(u_0, v_0)\in H^s(M)\times (H^s(M)\cap L^{\infty}(M))$ whenever $s>\frac{d}2-\frac12$, $d\geq 2$. For $d=2$, we apply a sharp version of the Gagliardo-Nirenberg inequality in compact manifold to derive an {\em a priori} estimate for the $H^1$-solution and use it to prove  the global well-posedness result in this space.
\end{abstract}

\maketitle

\section{Introduction}

In this work, we  consider the following initial value problem (IVP)  associated to the  Schrödinger-Debye (SD) system
\begin{equation}\label{SDE12}
\begin{cases}
i \partial_{t} u+ \Delta_{g} u = u v , \qquad{ } (t, x)  \in [0, \infty) \times M  \\
 \kappa \mbox{ } \partial_{t} v + v = \lambda |u|^{2}, \\
u(0,x) = u_{0}, \;\;v(0,x) = v_{0}, 
\end{cases}
\end{equation}
  where $M = (M^{d},g)$ is a complete Riemannian manifold of dimension $d \geq 2$, $\Delta_{g}$ denotes the Laplace-Beltrami operator on $M$,  $u=u(t,x)$ is a complex-valued function, $v=v(t,x)$ is a real valued  function, $\lambda = \pm 1$ and $\kappa > 0$ is a constant.  The cases $ \lambda = -1$  and $\lambda = 1$ are known as  focusing and defocusing nonlinearities, respectively.

This system models the propagation of  electromagnetic waves
through a nonresonant medium, whose nonlinear polarization lags behind the induced electric field
(see \cite{NEWELLMOLONEY} for more physical details).
In this context, a non-Euclidean metric would correspond
essentially to a medium with variable optical index (see \cite{GERARD2005EMS} for more details).
 Notice that, in the absence of the delay ($\kappa$ = 0) representing
an instantaneous polarization response, the system $(\ref{SDE12})$ reduces to the cubic nonlinear Schrödinger
 (NLS) equation
\begin{equation}\label{NLS}
\begin{cases}
i \partial_{t} u + \Delta_{g} u = \lambda |u|^{2}u , \qquad{ }  (t, x) \in \mathbb{R} \times M  \\
u(0,x) = u_{0}.
\end{cases}
\end{equation}

 For sufficiently regular data, the mass of the solution $u$ of the SD system $(\ref{SDE12})$ is conserved by the flow. More precisely, we have
\begin{equation}\label{MassConservation}
m(t):=\int_{M}|u(t, x)|^{2} dg = \int_{M}|u_{0}(x)|^{2} dg = m(0) .
\end{equation}

Unlike what happens with the NLS equation \eqref{NLS}
other conservation laws like energy  for the  SD system  \eqref{SDE12}  are not known. This is one of the main differences between the NLS equation and SD system.  However, we can prove (see Proposition \ref{PropertieOfSolution} below) the following relation involving the gradient term 
\begin{equation}\label{PseudoEnergy}
\frac{d}{dt}\left(\int_{M} |\nabla_{g} u(t,x)|_{g}^{2} + \int_{M} |u(t,x)|^{2} v(t,x)\right) = \frac{1}{\kappa} \left( - \int_{M} |u(t,x)|^{2} v(t,x) + \lambda \int_{M} |u(t,x)|^{4}\right),
\end{equation}
where $\nabla_{g}$ and $|\cdot|_{g}$ are respectively the gradient and the norm associated 
with the metric $g$, and $|\cdot|$ denotes the modulus of any complex number. 

Note that the second equation in \eqref{SDE12} is a simple ODE and can be solved explicitly to get
\begin{equation}\label{ODEvsolution}
v(t, x) = e^{-\frac{t}{\kappa}} v_{0}(x) + \frac{\lambda}{\kappa} \int_{0}^{t} e^{-\frac{(t - \ell)}{\kappa}} |u(\ell, x)|^{2} d\ell.
\end{equation}
Using the value of $v$ from \eqref{ODEvsolution}, the system \eqref{SDE12} can be decoupled to obtain an  integro-differential equation
\begin{equation}\label{IntegroDifferentialFormulation}
\begin{cases}
i \partial_{t} u + \Delta_{g} u = \displaystyle  e^{-\frac{t}{\kappa}} u v_{0}(x) + \frac{\lambda}{\kappa} u\int_{0}^{t} e^{-\frac{(t - \ell)}{\kappa}} |u(\ell, x)|^{2} d\ell, \qquad{ }   \\
u(0,x) = u_{0}(x).\\
\end{cases}
\end{equation}

Our interest here is to address the well-posedness issues for the IVP \eqref{SDE12} posed on  $d$-dimensional compact Riemannian manifold $M$ without boundary. For this, we will consider the integro-differential equation \eqref{IntegroDifferentialFormulation} and treat it as an evolution of $u$ from the NLS point of view.  In the following subsection  we review the previous results regarding the local and global well-posedness theory for $(\ref{SDE12})$  when $M = \mathbb{R}^{d}$ or $\mathbb{T}^{d}$. In Subsection 1.2 we describe the new results  obtained in this work.

\subsection{Overview of former results, when $M = \mathbb{R}^{d}$ or $\mathbb{T}^{d}$, $d \geq 1$.}
The well-posedness issues and other qualitative behavior for the IVP \eqref{NLS}  are extensively studied in the literature considering  $M$ as $ \R^d $ or $\T^d$, see for example \cite{GERARD2005EMS,BOURG93A,BOURG93B,GRIL00,KATO87,GIN96} and references therein. When $M$ is a compact manifold without boundary a pioneer work  for the  IVP \eqref{NLS} was carried out in  \cite{BGT} where the authors derived Strichartz estimate with loss to get local well-posedness result. At this point, we also mention the work of Blair et al \cite{BSS2008,BSS2012} where the authors studied similar problem considering compact manifold with boundary as well. There are further improvement and extension of these results  taking in account of  several geometrical considerations on $M$, see \cite{GERARD2005EMS,OIVAN2010,HTW2006} and references therein. 

The IVP \eqref{SDE12} is also widely studied in literature to deal with the well-posedness issues considering $M=\R^d$ or $\T^d$.  As far as we know, the first work that deals with the well-posedness of  the IVP \eqref{SDE12} for  initial data
$(u_{0}, v_{0})\in H^{s}(\mathbb{R}^{d}) \times H^{s}(\mathbb{R}^{d})$, $ d= 1, 2, 3$ was due to 
Bidégaray (\cite{BIDEGARAY1} and \cite{BIDEGARAY2}). More precisely, the author in \cite{BIDEGARAY2} proved the local well-posedness for data with regularity $s>\frac{d}2$. 
This result was obtained by a fixed-point procedure applied to the Duhamel
formulation for the integro-differential equation $(\ref{IntegroDifferentialFormulation})$. The author  also considered the cases $s=1$ and $s=0$  to obtain solution $u$  respectively in $ L^{\infty}([0, T],H^{1}(\mathbb{R}^{d}) )$ and  $C([0, T], L^{2}(\mathbb{R}^{d})) \cap L^{\frac{8}{d}}([0, T], L^{4}(\mathbb{R}^{d}))$ for sufficiently small $T$  using the Strichartz estimates for the unitary group $S(t) = e^{i t \Delta_{\mathbb{R}^{d}}}$ associated to the linear Schrödinger equation.   Later, Corcho and Linares \cite{CORLIN} improved these results by obtaining global well-posedness for $H^1(\R^d)\times H^1(\R^d) $ and $L^2(\R^d)\times L^2(\R^d)$ data. In one dimensional case they  also obtained local well-posedness for data $(u_{0}, v_{0})\in H^{s}(\mathbb{R}^{d}) \times H^{s}(\mathbb{R}^{d})$, $0<s<1$ with  an optimal application of the Strichartz's and Kato's smoothing  estimates.  When $\frac12\leq s\leq 1$, they extended this later result lowering the regularity requirement on $v_0$. Moreover, the authors in \cite{CORLIN} proved persistence property for the evolution of $v$ too.

 In dimensions $2$ and $3$ the local well-posedness result was further improved by Corcho,  Oliveira and Drumond Silva  \cite{CORCHOOLIVEIRASILVA}  using
   Bourgain's spaces for any data $(u_{0}, v_{0}) \in H^{s}(\mathbb{R}^{d}) \times H^{\ell}(\mathbb{R}^{d})$, with $s$ and $\ell$ satisfying
$$
\max \{0, s- 1 \} \leq \ell \leq \min\{2s, s + 1 \}.$$

 It should be noted that the authors in \cite{CORCHOOLIVEIRASILVA} considered  the original system \eqref{SDE12} in an equivalent  formulation 
\begin{equation}\label{SDS2}
\begin{cases}
\displaystyle  u(t,\cdot) = S(t)u_{0} - i \int_{0}^{t} S(t - t') u(t',\cdot) v(t', \cdot) dt',\\
\displaystyle v(t,\cdot) = v_{0} +\frac{\lambda}{\kappa} \int_{0}^{t} \big(|u(t', \cdot)|^{2} - v(t', \cdot)\big)  dt',
\end{cases}
\end{equation}
 and modified arguments used for the Zakharov system in \cite{GINITSUVELO} to adapt in their case to get local well-posedness result.
In the same work, considering $d = 2$, they also derived an \textit{a priori} estimate in the space $H^{1}(\R^d) \times L^{2}(\R^d)$ for the both focusing and
defocusing cases of $(\ref{SDE12})$, which allows them to extend the local solution  to all positive times. Furthermore, it is  commented that a possible
blow-up in $H^{1}(\R^d)\times H^{1}(\R^d)$ could occur for $\|\nabla v\|_{L^{2}}$. Recently in \cite{GAMBOACARVAJAL}
global well-posedness  result has been established for
data in $H^{s}(\R^d)\times L^{2}(\R^d)$, $ \frac{2}{3} < s \leq 1$ and for data in $H^{1}(\R^d)\times H^{1}(\R^d)$ $(d = 2)$. 

The first work that deals the well-posedness issues for the IVP \eqref{SDE12}  when $M$ is a compact manifold  was done in  \cite{ARBIETOMATHEUS} considering $M = \mathbb{T}^{d}$, the $d$-dimensional torus. In fact, the authors in \cite{ARBIETOMATHEUS}  considered a general non-linearity of the form $\lambda |u|^{\gamma-1}$ with $\gamma \geq 3$ in the second equation of \eqref{SDE12} and  proved  that the IVP associated to the integro-differential equation
\begin{equation}\label{Int-Diff-2}
\begin{cases}
i \partial_{t} u + \Delta_{\T^d} u = \displaystyle  e^{-\frac{t}{\kappa}} u v_{0}(x) + \frac{\lambda}{\kappa} u\int_{0}^{t} e^{-\frac{(t - \ell)}{\kappa}} |u(\ell, x)|^{\gamma-1} d\ell, \qquad{ }   \\
u(0,x) = u_{0}(x)
\end{cases}
\end{equation} 
is globally well-posed for data  $H^{s}(\mathbb{T}) \times H^{s}(\mathbb{T})$, $s\geq 0$, and for $d\geq 2$,  locally
well-posed in $H^{s}(\T^d) \times H^{s}(\T^d)$ whenever $s \geq0$ and $2 \leq \gamma - 1 < \frac{4}{d - 2s}$. Moreover, for cubic nonlinearity $\gamma = 3$,  they also proved that the IVP (\ref{SDE12}) is globally well-posed in $H^{1}(\T^2) \times H^{1}(\T^2)$.


\subsection{Main results in a general compact Riemannian manifold $(M, g)$, $d \geq 2$.}

  In this subsection, we present some results on local and global well-posedness to the IVP $(\ref{SDE12})$ posed in compact Riemannian manifolds. We will consider the integro-differential equation $(\ref{IntegroDifferentialFormulation})$ and use the
 fixed point argument in a suitable Banach space. For $s > \frac{d}{2}$, one can use  the Sobolev embedding $H^{s}(M) \hookrightarrow L^{\infty}(M)$ and obtain local well-posedness of the IVP  \eqref{IntegroDifferentialFormulation} for given data  $(u_{0}, v_{0}) \in H^{s}(M) \times H^{s}(M)$. Therefore, our interest here is to consider the well-posedness issues for the IVP  \eqref{IntegroDifferentialFormulation} for given data with regularity below $\frac{d}2$. Recall that, if the Sobolev regularity of the given data is below $\frac{d}2$, one generally uses Strichartz inequality to perform contraction mapping principle. However, there is an extra difficulty in our case because the exact analogue of the Euclidean Strichartz
inequality does not hold in the compact Riemannian manifold (in fact, there is a loss of $\frac{1}{p}$ derivatives). Also, the Strichartz-like inequalities holds only  locally in time (see Section $2.2$ and \cite{BGT} for more details). We will follow an approach used in \cite{BGT} to obtain some well-posedness results for the  $(\ref{IntegroDifferentialFormulation})$ imposing some extra condition on the initial data $v_0$. 
As far as we know, the well-posedness results obtained in this work are the first  for the IVP \eqref{SDE12} associated to the  SD system  posed on a general compact manifold of dimension $d \geq 2$.  Before announcing the main results we record definitions of  the $d$-admissible pair and well-posedness of the IVP that  will be used throughout this work.

\begin{defini}
A pair $(p, q)$ is  called $d$-admissible ($d \geq 2$) if
\begin{equation}\label{pqPairCond}
\frac{2}{p} + \frac{d}{q} = \frac{d}{2}
\end{equation}
with $2\leq p, q\leq \infty $ and $(p, q, d) \neq (2, \infty, 2)$.
\end{defini}

\begin{defini}\label{LWPDef}(\cite{GERARD2005EMS})
We say that the IVP $(\ref{IntegroDifferentialFormulation})$ is locally well-posed in the space $H^{s}(M)$ if,
for all bounded subset $B$ of $H^{s}(M)$, there exist $T > 0$ and a Banach space $X_{T}$ continuously embedded
in $C([0, T], H^{s}(M))$ such that:
\begin{enumerate}
\item[$i)$] For any initial data $u_{0} \in B$, $(\ref{IntegroDifferentialFormulation})$ has a unique solution
$u \in X_{T}$ such that $u(0) = u_{0}$.

\item[$ii)$] If $u_{0} \in H^{r}(M)$ for $r > s$, then $u \in C([0, T], H^{r}(M))$.

\item[$iii)$] The map $u_{0} \in B \mapsto u \in X_{T}$ is continuous.
\end{enumerate}

If the properties  $i),ii)$ and $iii)$ above hold true for any time $T > 0$, we say that the IVP $(\ref{IntegroDifferentialFormulation})$ is globally well-posed in $H^{s}(M)$.
\end{defini}


Now we move to state the first main result of this work which deals with  the local well-posedness of the IVP \eqref{SDE12}.
\begin{teorema}\label{TeorIntro5}
Let $M$ be a compact Riemannian manifold of dimension $d \geq2$, $p > 2$ satisfying \eqref{pqPairCond} and 
 $s > \frac{d}{2} - \frac{1}{p}$. Let $v_{0} \in H^{s}(M) \cap L^{\infty}(M)$ be fixed. Then for any and $u_0\in B\subset  H^s(M)$, $B$ bounded,  the IVP \eqref{IntegroDifferentialFormulation} is locally well-posed, i.e., there exist a time
$T = T(\kappa,\|v_{0}\|_{H^{s} \cap L^{\infty}}, \|u_{0}\|_{H^{s}}) > 0,$
and a Banach space
\[X_{T} := C([0,T]; H^{s}(M)) \cap L^{p}([0,T];L^{\infty}(M)),
\]
 such that the conditions i), ii) and iii) of Definition \ref{LWPDef} are satisfied. Moreover, the application $u_{0} \in B \mapsto u \in X_{T}$ is Lipschitz continuous.
\end{teorema}

As discussed above, to prove this theorem we use the contraction mapping principle on integral formulation as was done in the Euclidean as well as $\T^d$ case. However, in our context, the nonlinear term involving $uv_0$ in the integral formulation \eqref{IntegroDifferentialFormulation} does not behave well when applying  the Strichartz estimate with loss
$$\|e^{it\Delta }u_0\|_{L^{p}(I, L^{q}(M))} \leq C(I) \| u_{0}\|_{H^{\frac{1}{p}}(M)}.$$ 
This fact compelled us to impose an extra condition $v_{0} \in H^{s}(M) \cap L^{\infty}(M)$ on the initial data $v_0$ (see the estimates \eqref{EQ18}, \eqref{EQ19} below). 
We also show  that this additional condition on the initial data $v_0$ is preserved by the evolution $v$ during the time of existence, see Remark \ref{Vpersistence} below. In this sense, this extra condition on initial data is not so unusual.

We also prove the following persistence property for the solution $v$.
\begin{corolario}\label{v(t)Persistence}
Consider the expression of the function  $v$ given by  \eqref{ODEvsolution}. If $v_{0} \in H^{s} \cap L^{\infty}$ and $u_0\in H^s(M)$,  then
\[ 
v \in  C([0,T]; H^{s}(M)) \cap L^{p}([0,T];L^{\infty}(M)),
\]
where $s,p$ and $T$ are as in Theorem \ref{TeorIntro5}.
\end{corolario}

From  Theorem \ref{TeorIntro5}, we observe that if $d\geq 3$, the regularity requirement on the initial data to obtain local well-posedness is  $s > \frac{d}{2} - \frac{1}{p}>1$. However, for $d=2$, we obtain local well-posedness in $H^{1}(M^d)$ because is this case the regularity requirement is $s > 1- \frac{1}{p}$,  which can go below $1$ for  $p>2$.  This sort of situation  appears in the case of NLS equation as well.  Now the natural question is, whether one can prove the global well-posedness result in $H^1(M^2)$ for the SD system?

Generally, conserved quantities are the main ingredients in the proof the global well-posedness results. In this case, we do not have energy conservation law. As discussed above, we derive a relation involving the gradient of the solution $u(t, \cdot)$ (see Proposition  \ref{PropertieOfSolution} below)  and use a sharp version of the Gagliardo-Nirenberg inequality in the context of compact Riemannian manifold to obtain an {\em a priori} estimate
\[
\|u\|_{L^{\infty}((0, T);H^1(M))}\lesssim \|u_0\|_{H^1(M)}+\|u_0\|_{L^2(M)}\|v_0\|_{L^2(M)}, 
\]
for some $T >0$ (see Proposition \ref{PropAPrioriEstimate} below). We use this {\em a priori} estimate to provide an affirmative answer to the question posed above. More precisely, we prove the following global well-posedness result in dimension $2$.

\begin{teorema}\label{TeorIntro7} 
 Let $u_{0} \in H^{1}(M^{2})$,  $v_{0} \in H^{1+}(M^{2})$ and $u \in X_{T}:= C([0,T], H^{1}(M^{2})) \cap L^{p}([0,T], L^{\infty}(M^{2}))$ be the  local solution to the integro-differential equation \eqref{IntegroDifferentialFormulation} obtained in Theorem \ref{TeorIntro5}. Then $u$ is global solution, that is, $u \in C([0,T), H^{1}(M^{2})) \cap L^{p}([0,T), L^{\infty}(M^{2}))$ for any bounded $T > 0$.
\end{teorema}

Also, considering the global solution given by Theorem \ref{TeorIntro7} we obtain an \textit{estimative on the growth} to the  $H^{1}$-norm of the solution in the case $\lambda = 1$ and  $v_{0} \geq 0$, showing that the growth is at most exponential, see Subsection \ref{sub-sec4.3} for details.

We finish this section by providing the structure of this article. In Section \ref{sec-2} we introduce function spaces and some preliminary results including Strichartz estimate. In Section \ref{sec-3} we provide a Proof for Theorem~\ref{TeorIntro5}.  Section \ref{sec-4} is devoted to derive an {\em a priori} estimate and to prove Theorem \ref{TeorIntro7}. Finally, in Section \ref{sec-5} we record some concluding remarks and future works. 

\section{ Definitions and  Preliminary results } \label{sec-2}

In this section we introduce some notations and function  spaces which are used throughout this work and establish some properties. Also we derive some preliminary estimates on these spaces.

\subsection{Notation} We start introducing some notations. To make exposition simple we denote the Beltrami-Laplace operator associated to the metric $g$ by $\Delta := \Delta_{g}$ and the gradient by $\nabla:=\nabla_{g}$. We use notation  $A \lesssim B$ to say that there exists a constant $C$ 
such that $A \leq C B$ and $A \simeq B$ to say both $A \lesssim  B$ and $B \lesssim A$.
The notation  $a+$ means $a + \epsilon$ for  $0 < \epsilon \ll 1$.
Given a Banach space $X$, a measurable function
$u: I \subset \mathbb{R} \rightarrow X$, and an exponent $p \in [1, \infty]$, we denote
\[
\|u\|_{L^{p}(I,X)} = \left[ \int_{I} (\|u(t)\|_{X})^{p} dt \right]^{\frac{1}{p}},
\]
if $p \in [1, \infty)$ and $\|u\|_{L^{\infty}(I,X)} = \esssup_{t \in I}\|u(t)\|_{X} $.

\subsection{Sobolev Spaces on Compact Riemannian Manifolds}
Given a $d$-dimensional compact Riemannian manifold $M:=(M^d, g)$, we can consider a finite atlas $\mathcal{A} = (U_{\alpha}, \kappa_{\alpha})_{\alpha =1}^{k}$ and a partition
of unity $(h_{\alpha})_{\alpha = 1}^{k}$ on $M$ subordinate to the finite covering
  $ \{ U_{\alpha} \}_{\alpha = 1}^{k}$, i.e., satisfying
$\supp (h_{\alpha}) \subset U_{\alpha}$, $0 \leq h_{\alpha} \leq 1$ and $\sum_{\alpha} h_{\alpha} = 1$ on $M$. With this setting, we define the Sobolev space  $H^{s}(M) : = W^{s, 2}(M)$  of order $s\geq 0$ as being the completion of the space of smooth functions $\mathcal{D}(M) : = \{f: M \rightarrow \mathbb{C}: f \in C^{\infty}(M) \}$ with respect to the norm
\begin{equation}\label{SobolevNormaHs}
 \|f\|_{H^{s}(M)} = \left(\sum_{\nu} \|(h_{\nu} f) \circ \kappa_{\nu}\|_{H^{s}(\mathbb{R}^{d})}^{2} \right)^{\frac{1}{2}}.
\end{equation}
Note that considering  $\kappa_{\nu} : B(x_{\nu}, \tilde{r}) \subset \mathbb{R}^{d}\rightarrow B(y_{\nu}, r) \subset M $, we have
\[
 (h_{\nu} f) \circ \kappa_{\nu} :B(x_{\nu}, \tilde{r}) \subset \mathbb{R}^{d} \rightarrow \mathbb{C}. 
\] 
 Hence, $\supp (h_{\nu} f) \subset B(y_{\nu}, r) \subset M$, and we have that $(h_{\nu} f) \circ \kappa_{\nu}$ possesses support in 
\[
 \kappa_{\nu}^{-1} (\supp (h_{\nu} f) ) \subset B(x_{\nu}, \tilde{r}) \subset \mathbb{R}^{d}. 
\] 
 Thus, we can assume $(h_{\nu} f) \circ \kappa_{\nu}$  be extended from the corresponding chart on
$\mathbb{R}^{d}$ by zero outside of its support and the norm is calculated in $\mathbb{R}^{d}$ .

Note that, from the theory of Sobolev spaces on the Euclidean space  $\mathbb{R}^{d}$,  the following inequality holds.
\begin{lema}
Let $s>0$ and  $f, g \in L^{\infty}(\mathbb{R}^{d}) \cap H^{s}(\mathbb{R}^{d})$, then
\begin{equation}\label{EQ3}
\| fg \|_{H^{s}(\mathbb{R}^{d})} \leq C ( \| f \|_{H^{s}(\mathbb{R}^{d})} \| g \|_{L^{\infty}(\mathbb{R}^{d})} + \| g \|_{H^{s}(\mathbb{R}^{d})} \| f \|_{L^{\infty}(\mathbb{R}^{d})}).
\end{equation}
\end{lema}
\begin{proof}
See \cite{ALINHACGERARD} pg. 84.
\end{proof}

Now, we use $(\ref{EQ3})$ to obtain an analogous estimate in compact Riemannian manifolds which  will be important in our analysis.
\begin{lema}\label{EQ1}
Let $s>0$ and $f, g \in H^{s}(M) \cap L^{\infty}(M)$, then there exists a constant  $C > 0$ such that
\begin{equation}\label{BilinearHs}
\|fg\|_{H^{s}(M)} \leq  C (\|f\|_{H^{s}(M)} \|g\|_{L^{\infty}(M)} + \|f\|_{L^{\infty}(M)} \|g\|_{H^{s}(M)}).
\end{equation}

\end{lema}
\begin{proof} Using definition $(\ref{SobolevNormaHs})$, we have
\begin{equation}\label{EQ2}
\|f g\|_{H^{s}(M)} = \left(\sum_{\nu} \|(h_{\nu} f g) \circ \kappa_{\nu}\|_{H^{s}(B(x_{\nu}, \tilde{r}))}^{2} \right)^{\frac{1}{2}}
 = \left(\sum_{\nu} \|(h_{\nu} f)\circ \kappa_{\nu} \cdot g  \circ \kappa_{\nu}  \|_{H^{s}(B(x_{\nu}, \tilde{r}))}^{2} \right)^{\frac{1}{2}}.
\end{equation}

Taking into account the discussion after definition \eqref{SobolevNormaHs}, and  in view of the localization of the support of the function $(h_{\nu} f g) \circ \kappa_{\nu}$, we can replace  
$\mathbb{R}^{d}$ by  $B(x_{\nu}, \tilde{r})$ in the estimate \eqref{BilinearHs}. 
With this consideration, applying $(\ref{EQ3})$ in $(\ref{EQ2})$, it follows that
\begin{equation}\label{EQ4}
\begin{split}
  \|f g\|_{H^{s}(M)} &\leq \left(\sum_{\nu}  \|(h_{\nu} f)\circ \kappa_{\nu}\|_{H^{s}(B(x_{\nu}, \tilde{r}))}^{2} \|g  \circ \kappa_{\nu} \|_{L^{\infty}(B(x_{\nu}, \tilde{r}))}^{2} \right)^{\frac{1}{2}}  \\
  &\qquad + \left( \sum_{\nu} \|(h_{\nu} f)\circ \kappa_{\nu}\|_{L^{\infty}(B(x_{\nu}, \tilde{r}))}^{2} \|g  \circ \kappa_{\nu} \|_{H^{s}(B(x_{\nu}, \tilde{r}))}^{2}\right)^{\frac{1}{2}}.
\end{split}
\end{equation}

Now, using that $\sup_{\nu}\|g  \circ \kappa_{\nu} \|_{L^{\infty}(B(x_{\nu}, \tilde{r}))} = \|g\|_{L^{\infty}(M)}$ in $(\ref{EQ4})$, we obtain
\begin{equation}\label{EQ5}
\begin{split}
  \|f g\|_{H^{s}(M)} &\leq \left(\sum_{\nu}  \|(h_{\nu} f)\circ \kappa_{\nu}\|_{H^{s}(B(x_{\nu}, \tilde{r}))}^{2}  \right)^{1/2} \|g\|_{L^{\infty}(M)}  \\
  & \qquad + \left( \sum_{\nu} \|(h_{\nu} f)\circ \kappa_{\nu}\|_{L^{\infty}(B(x_{\nu}, \tilde{r}))}^{2} \|g  \circ \kappa_{\nu} \|_{H^{s}(B(x_{\nu}, \tilde{r}))}^{2}\right)^{\frac{1}{2}}.
\end{split}
\end{equation}

Using the property of the partition of unity, we have  $\supp(h_{\nu} g) \subset B(y_{\nu}, r)$, which in turn implies
\begin{equation}\label{EQ6}
 \|g  \circ \kappa_{\nu} \|_{H^{s}(B(x_{\nu}, \tilde{r}))}\leq \sum_{\sigma} \|( h_{\sigma} g)  \circ \kappa_{\nu} \|_{H^{s}(B(x_{\nu}, \tilde{r}))} = \|( h_{\nu} g)  \circ \kappa_{\nu} \|_{H^{s}(B(x_{\nu}, \tilde{r}))} \leq \|g\|_{H^{s}(M)}.
\end{equation}

Thus, using that $ \sup_{\nu} \|h_{\nu} f\circ \kappa_{\nu}\|_{L^{\infty}(B(x_{\nu}, \tilde{r}))}  \leq \| f\|_{L^{\infty}(M)}$, $0 \leq h_{\nu}(x) \leq 1$ for all $x \in M$, and the relation $(\ref{EQ6})$ in $(\ref{EQ5})$, one obtains
\begin{equation}
\begin{split}
  \|f g\|_{H^{s}(M)} &\leq \|f\|_{H^{s}(M)} \|g \|_{L^{\infty}(M)}
  + \left( \sum_{\nu} \|h_{\nu}\|^{2}_{L^{\infty}(M)}\| f\|_{L^{\infty}(M)}^{2} \right)^{\frac{1}{2}} \|g \|_{H^{s}(M)}\\
& \leq C (\|f\|_{H^{s}(M)} \|g \|_{L^{\infty}(M)}  + \|f\|_{L^{\infty}(M)} \|g \|_{H^{s}(M)}).
\end{split}
\end{equation}
\end{proof}

 Recall  that, as  the operator $-\Delta$ is self-adjoint and has a purely discrete  spectrum contained in $[0, \infty)$, we can construct spectral multipliers $\varphi(-\Delta)$ for any measurable function $\varphi: [0, \infty) \rightarrow \mathbb{C}$ of at most
polynomial growth, see page 228 in \cite{BKKBOOK}. In particular, for $s\in \R$ we can define fractional powers $(- \Delta)^{s/2}$ and $(1 - \Delta)^{s/2}$, as well as Schrödinger propagators 
$e^{- i t \Delta}$ and Littlewood-Paley type operators
on $- \Delta$. These spectral multipliers commute with each other, and are bounded in $L^{2}$
if their respective  symbols $\varphi$ are bounded (for example, $\|\varphi(-\Delta)\|_{L^{2} \rightarrow L^{2}} \leq \|\varphi\|_{L^{\infty}}$ if $\varphi \in L^{\infty}(\mathbb{R}_{+})$).
We can use these properties of $-\Delta$ described above, for $\sigma>0$ and $q\geq 1$  to define the generalized Sobolev space $W^{\sigma, q}(M)$ of order $\sigma$ based on $L^{q}$, 
with norm
\begin{equation}\label{Wsq}
 \|f\|_{W^{\sigma, q}(M)} : = \|(1 - \Delta)^{\frac{\sigma}{2}}f\|_{L^{q}(M)}.
\end{equation}

Note that $W^{\sigma, 2}(M) = H^{\sigma}(M)$ and the norm defined in \eqref{SobolevNormaHs} is equivalent to the following
\begin{equation}\label{HsNormEquivalences}
\|f\|_{H^{s}(M)}  \approx  \|f\|_{L^{2}(M)} + \|(-\Delta)^{s/2}f\|_{L^{2}(M)} \approx \|(1 - \Delta)^{\frac{s}{2}}f\|_{L^{2}(M)}, \quad s\geq 0.
\end{equation}
For details of this equivalence we refer to   \cite{MIZUTANITZVETKOV} and  \cite{IPESENSON}.  


\subsection{Strichartz estimates with loss and applications}
In this subsection, we quickly describe the Strichartz estimate with a loss of derivative on compact manifolds obtained in \cite{BGT} and apply that to obtain some useful estimates which will be in the following sections. We start with the following key estimate.

\begin{lema}\label{StrichartzBGT2004}(Strichartz estimate with  loss of derivatives)
Let $(M,g)$ be a compact Riemannian manifold of dimension $d \geq 2$. Then, the solution  $u = e^{it \Delta}u_{0}$ of the linear Schrödinger equation, satisfies, for any finite time interval  $I$, and $d$-admissible pair $(p,q)$ with $q < \infty$,
\begin{equation}\label{StrichartzInequality}
\|u\|_{L^{p}(I, L^{q}(M))} \leq C(I) \| u_{0}\|_{H^{\frac{1}{p}}(M)}.
\end{equation}
\end{lema}

 A detailed proof of this lemma can be found in \cite{BGT}. The main ingredients in the proof are the local dispersion estimate  \cite{BGT} and Keel-Tao Lemma  \cite{KEELTAO}.

In what follows, we use the  Strichartz estimate  $(\ref{StrichartzInequality})$  to establish some  linear and nonlinear estimates which we  will apply  to obtain local well-posedness results.

\begin{lema} If $(p,q)$ is a $d$-admissible pair, and $q < \infty$, then
\begin{equation}\label{StrichartzInequality2}
\left \| \int_{0}^{t} e^{i ( t - \tau) \Delta} f(\tau) d\tau  \right \|_{L^{p}([0, T], L^{q}(M))} \leq C_{T} \| f \|_{L^{1}([0, T], H^{\frac{1}{p}}(M))}.
\end{equation}
\end{lema}
\begin{prova} A detailed  proof of this Lemma can be found in  \cite{BGT}, (see Corollary 2.10 there). The main idea in the proof is as follows. Defining $F_{t}(\tau) := \chi_{[0,t]}(\tau) e^{i (t - \tau) \Delta} f(\tau)$ and using the Minkowski's inequality, one can obtain
\begin{equation}
\begin{split}
\left \| \int_{0}^{t} e^{i ( t - \tau) \Delta} f(\tau) d\tau  \right \|_{L^{p}([0, T], L^{q}(M))}  \leq    \int_{0}^{T} \left\| F_{t}(\tau)\right\|_{L^{p}([0, T], L^{q}(M))} d\tau.
\end{split}
\end{equation}

Now, using $(\ref{StrichartzInequality})$, we get
\[
 \int_{0}^{T} \left\| F_{t}(\tau)\right\|_{L^{p}([0, T], L^{q}(M))} d\tau \leq   \int_{0}^{T} \left\| e^{i (t - \tau) \Delta} f(\tau) \right\|_{L^{p}([0, T], L^{q}(M))} d\tau \leq C_{T}  \int_{0}^{T} \left\| f(\tau) \right\|_{H^{\frac{1}{p}}} d\tau,
\]
which proves the lemma.
\end{prova}

\begin{lema}\label{StrichartzInequality3Estimate} Let  $(p, q)$ be a  $d$-admissible pair, and $\sigma := s - \frac{1}{p}$, where $s > \frac{d}{2} - \frac{1}{p}$ ($p > 2$). Then, we have
\begin{equation}\label{StrichartzInequality3}
\| e^{i t \Delta} u_{0} \|_{L^{p}([0, T], W^{\sigma, q}(M))}= \| (1- \Delta)^{\frac{\sigma}{2}} e^{i t \Delta} u_{0} \|_{L^{p}([0, T], L^{q}(M))} \leq C_{T} \|u_{0}\|_{H^{s}(M)}.
\end{equation}
\end{lema}
\begin{dem}
Using the fact that  the operators $(1- \Delta)^{\frac{\sigma}{2}}$ and $ e^{i t \Delta}$ commute, we have
\[
\| (1- \Delta)^{\frac{\sigma}{2}} e^{i t \Delta} u_{0} \|_{L^{p}([0, T], L^{q}(M))} = \| e^{i t \Delta} (1- \Delta)^{\frac{\sigma}{2}}  u_{0} \|_{L^{p}([0, T], L^{q}(M))}.
\]

Next, applying the  Strichartz estimate (\ref{StrichartzInequality}), we get
\[
\| e^{i t \Delta} (1- \Delta)^{\frac{\sigma}{2}}  u_{0} \|_{L^{p}([0, T], L^{q}(M))} \lesssim_{T} \|(1- \Delta)^{\frac{\sigma}{2}}  u_{0}\|_{H^{1/p}(M)}.
\]

By definition, one has
\[
\|(1- \Delta)^{\frac{\sigma}{2}}  u_{0}\|_{H^{1/p}(M)} = \|(1- \Delta)^{\frac{1}{2p}}(1- \Delta)^{\frac{\sigma}{2}}  u_{0}\|_{L^{2}(M)}.
\]

Finally, applying the semigroup property,  $(1- \Delta)^{\alpha}(1- \Delta)^{\beta} = (1- \Delta)^{ \alpha + \beta}$, we obtain
\begin{equation}
\begin{split}
\|(1- \Delta)^{\frac{1}{2p}}(1- \Delta)^{\frac{\sigma}{2}}  u_{0}\|_{L^{2}(M)} &= \|(1- \Delta)^{\frac{1}{2}(\frac{1}{p} + \sigma)}  u_{0}\|_{L^{2}(M)}  \\
& \lesssim \|u_{0}\|_{H^{\frac{1}{p} + \sigma}(M)}
 =\|u_{0}\|_{H^{s}(M)}.
\end{split}
\end{equation}
\end{dem}

\begin{lema}\label{StrichartzInequality4Estimate}
Let $(p,q)$ be a $d$-admissible pair and  $\sigma:= s - \frac{1}{p} > \frac{d}{q}$. Then, we have

\begin{equation}\label{StrichartzInequality4}
\left\|\int_{0}^{t} e^{i (t - \tau) \Delta} F(u(\tau)) d \tau \right\|_{L^{p}([0, T], W^{\sigma,q}(M))} \leq C_{T}   \| F(u)\|_{ L^{1}([0,T];H^{s}(M))},
\end{equation}
where $s > \frac{1}{p} + \frac{d}{q} = \frac{d}{2} - \frac{1}{p}$.
\end{lema}
\begin{dem}
The proof follows from  the Lemma $\ref{StrichartzInequality3Estimate}$, so we omit details.
\end{dem}

We close this section with  the following technical lemma which  will be used in our analysis later.

\begin{lema}\label{EQ21Star}
Consider the expression
 \begin{equation}\label{EQ21}
I(t) := \left\| f(t) \int_{0}^{t} e^{- \frac{(t - t')}{\kappa}} g(t') h(t') dt' \right\|_{H^{s}}.
\end{equation}
Then, one has
\begin{equation}\label{EQ21-1}
\begin{split}
I(t) &\leq \left\| f(t) \right\|_{H^{s}}    \left\| g\right\|_{L^{p}_{T}L^{\infty}_{x}} \left\| h \right\|_{L^{p}_{T}L^{\infty}_{x}}  + \left\| f(t) \right\|_{L^{\infty}_{x}}  \| h \|_{L_T^{\infty}H^s}   \| g\|_{L_T^pL_x^{\infty}} T^{\gamma_{p}}\\
&\quad +  \left\| f(t) \right\|_{L^{\infty}_{x}}  \| g\|_{L_T^{\infty}H^s} \| h \|_{L_T^pL_x^{\infty}} T^{\gamma_{p}},
\end{split}
\end{equation}
where $\gamma_{p}:= 1 - \frac{1}{p}$ $(p > 2)$. In particular if $\|\cdot\|_{X_{T}} := \|\cdot\|_{L^{\infty}([0, T];H^{s}(M))} + \|\cdot\|_{L^{p}([0, T]; L^{\infty}(M))}$, then
\begin{equation}\label{EQ21a}
I(t) \lesssim \left\| f(t) \right\|_{H^{s}}    \left\| g\right\|_{X_{T}} \left\| h \right\|_{X_{T}} + 2  \left\| f(t) \right\|_{L^{\infty}_{x}}  \| g\|_{X_{T}} \| h \|_{X_{T}} T^{\gamma_{p}}.
\end{equation}

\end{lema}

\begin{proof} First, we apply the Lemma $\ref{EQ1}$ to $(\ref{EQ21})$ and obtain
\begin{equation}\label{EQ22}
\begin{split}
I(t) & \lesssim \left\| f(t) \right\|_{H^{s}}   \left\|\int_{0}^{t} e^{- \frac{(t - t')}{\kappa}} g(t') h(t') dt' \right\|_{L^{\infty}_{x}}  
 + \left\| f(t) \right\|_{L^{\infty}_{x}}   \left\|\int_{0}^{t} e^{- \frac{(t - t')}{\kappa}} g(t') h(t') dt' \right\|_{H^{s}}.
\end{split}
\end{equation}

Using Minkowski's inequality for integrals it follows from $(\ref{EQ22})$ that
\begin{equation}\label{EQ23}
\begin{split}
 I(t) &\lesssim \left\| f(t) \right\|_{H^{s}}  \int_{0}^{t} e^{- \frac{(t - t')}{\kappa}} \left\| g(t')\right\|_{L^{\infty}_{x}} \left\| h(t') \right\|_{L^{\infty}_{x}}  dt'\\
& \quad+ \left\| f(t) \right\|_{L^{\infty}_{x}}   \int_{0}^{t} e^{- \frac{(t - t')}{\kappa}} \left\| g(t') h(t') \right\|_{H^{s}} dt'.
\end{split}
\end{equation}

Since $e^{- \frac{(t - t')}{\kappa}} \leq 1$ for $0 \leq t' \leq t < T \leq 1 $, we infer  from $(\ref{EQ23})$ that
\begin{equation}\label{EQ24}
\begin{split}
I(t) & \lesssim \left\| f(t) \right\|_{H^{s}}  \int_{0}^{T}  \left\| g(t')\right\|_{L^{\infty}_{x}} \left\| h(t') \right\|_{L^{\infty}_{x}}  dt'\\
& \quad+ \left\| f(t) \right\|_{L^{\infty}_{x}}  \| h \|_{L^{\infty}_{t}H^{s}} \int_{0}^{T}  \| g(t',\cdot)\|_{L^{\infty}_{x}} dt' +  \left\| f(t) \right\|_{L^{\infty}_{x}}  \| g\|_{L^{\infty}_{t}H^{s}} \int_{0}^{T}\| h(t',\cdot) \|_{L^{\infty}_{x}} dt' .
\end{split}
\end{equation}

Next, noting that $p > 2$ we can apply  Hölder's inequality in $(\ref{EQ24})$, we get
\begin{equation}\label{EQ25}
\begin{split}
I(t) & \lesssim \left\| f(t) \right\|_{H^{s}}    \left\| g\right\|_{L^{2}_{T}L^{\infty}_{x}} \left\| h\right\|_{L^{2}_{T}L^{\infty}_{x}}+ \left\| f(t) \right\|_{L^{\infty}_{x}}  \| h \|_{L^{\infty}_{T}H^{s}}   \| g\|_{L^{p}_{T} L^{\infty}_{x}} T^{\gamma_{p}}  \\
&\quad+  \left\| f(t) \right\|_{L^{\infty}_{x}}  \| g\|_{L^{\infty}_{T}H^{s}} \| h \|_{L^{p}_{T}L^{\infty}_{x}} T^{\gamma_{p}}\\
&\leq \left\| f(t) \right\|_{H^{s}}    \left\| g\right\|_{L^{p}_{T}L^{\infty}_{x}} \left\| h \right\|_{L^{p}_{T}L^{\infty}_{x}}  + \left\| f(t) \right\|_{L^{\infty}_{x}}  \| h \|_{L_T^{\infty}H^s}   \| g\|_{L_T^pL_x^{\infty}} T^{\gamma_{p}}\\
&\quad +  \left\| f(t) \right\|_{L^{\infty}_{x}}  \| g\|_{L_T^{\infty}H^s} \| h \|_{L_T^pL_x^{\infty}} T^{\gamma_{p}},
\end{split}
\end{equation}
which proves \eqref{EQ21-1}. The estimate \eqref{EQ21} follows from \eqref{EQ25}  using  the definition of $\|.\|_{X_{T}}$-norm.
\end{proof}

\section{Integral Formulation and Local Well-Posedness}\label{sec-3}
In this section we prove the local well-posedness result stated  in Theorem \ref{TeorIntro5}. As in the Euclidean case (see \cite{BIDEGARAY2, CORLIN}), here also we consider the integro-differential formulation  \eqref{IntegroDifferentialFormulation} of the IVP \eqref{SDE12}. 

For given $u_{0} \in  H^{s}(M)$ we use the usual Duhamel formula to write \eqref{IntegroDifferentialFormulation} in an equivalent integral formulation 
 \begin{equation}\label{SDS1234}
u(t) = S(t) u_{0} -  i \int_{0}^{t} S(t - \ell) G(\ell) d \ell, \qquad{ }   (t, x) \in [0, \infty) \times  M, \\
\end{equation}
 where $S(t) = e^{it \Delta}$ is the  group associated to the linear Schrödinger equation, and
 \begin{equation}\label{Wintegral}
 G(\ell) :=  e^{- \frac{\ell}{\kappa}} u(\ell) v_{0}(x) +  \frac{\lambda}{\kappa} u(\ell) \int_{0}^{\ell} e^{- \frac{(\ell - t')}{\kappa}} |u(t')|^{2} d t' =: G_{1}(\ell) + G_{2}(\ell).
 \end{equation}
 In order to prove a local existence result, we will use the classical contraction principle in  an appropriate Banach space $X_{T} \subset C([0, T], H^{s}(M))$. In sequel, we use the estimates proved in the previous section to provide details of this argument.

\begin{proof}[Proof of Theorem \ref{TeorIntro5}] Let $d \geq 2$. Assume $s < \frac{d}{2}$ and select $p > 2$ such that $s>\frac{d}{2}-\frac{1}{p}$. Also let  $v_0\in H^s(M)\cap L^{\infty}(M)$ be fixed and  $u_{0} \in  H^{s}(M)$ be given. Let us define a function   space
\[
X_{T} :=  \{ u \in C([0, T]; H^{s}(M)) \cap L^{p}([0, T]; L^{\infty}(M)): \|u\|_{X_{T}} < \infty \},
\] 
where
\begin{equation}\label{XTnorm}
\|u\|_{X_{T}} := \|u\|_{L^{\infty}([0, T];H^{s}(M))} + \|u\|_{L^{p}([0, T]; L^{\infty}(M))}.
 \end{equation}

We have that $X_{T} \subset C([0, T], H^{s}(M))$ is complete. Taking into account the Duhamel's formula \eqref{SDS1234} we define and application $\Phi: X_{T} \rightarrow X_{T}$ by
\begin{equation}\label{Duhamel-1}
\Phi (u)(t) = S(t) u_{0} - i \int_{0}^{t} S(t - \ell) G(\ell) d \ell,
\end{equation}
where $G=G_1+G_2$ is given by \eqref{Wintegral}.
We will show that there exist a  time $T>0$ and $R>0$ such that the application $\Phi$  is a contraction on the ball $B_{T}^{R}\subset X_{T}$, given by
\[
B_{T}^{R} =  \{ u \in X_{T}: \|u\|_{X_{T}} \leq R \}.
\]

Using that $S(t)$ is an isometry in $H^{s}(M)$ and the Minkowski's inequality we get
\begin{equation}\label{EQ10a}
\begin{split}
\|\Phi(u(t))\|_{H^{s}}  &\leq \|S(t)u_{0} \|_{H^{s}} + \left\|\int_{0}^{t} S(t - \ell) G(\ell)
    d\ell \right\|_{H^{s}} \\
&\leq \|u_{0}\|_{H^{s}} + \int_{0}^{t}\|G(\ell)  \|_{H^{s}} d\ell.
\end{split}
\end{equation}
So, we have
\begin{equation}\label{EQ11}
\|\Phi(u) \|_{L^{\infty}_{T}H^{s}}
\leq \|u_{0} \|_{H^{s}} + \|G \|_{L^{1}([0, T]; H^{s}(M))}.
\end{equation}
In order to bound the term $\|\Phi(u)\|_{L^{p}([0, T], L^{\infty}(M))}$ we use 
the Sobolev embedding 
\[
 W^{\sigma, q}(M) \hookrightarrow L^{\infty}(M),
\]
for  $\sigma:=s-\frac1p > \frac{d}{q}$  with   $(p,q)$ satisfying \eqref{pqPairCond}.  Thus, there exists $C > 0$ such that 
\begin{equation}\label{EQ11.1}
\|\Phi(u)\|_{L^{p}([0, T], L^{\infty}(M))} \leq C \|\Phi(u) \|_{L^{p}([0, T], W^{\sigma, q}(M))} .
\end{equation}
 Now,  taking the $\|\cdot\|_{L^{p}([0,T];W^{\sigma, q})}$-norm in \eqref{Duhamel-1}, one obtains
\begin{equation}\label{EQ12}
\begin{split}
\|\Phi(u)\|_{L^{p}([0,T];W^{\sigma, q})}  & \leq \|S(t) u_{0}\|_{L^{p}([0,T];W^{\sigma, q})}+ \left\|\int_{0}^{t} S(t - \ell) G(\ell) d\ell \right\|_{L^{p}([0,T];W^{\sigma, q})}.
\end{split}
\end{equation}

Applying the estimates $(\ref{StrichartzInequality3})$ and $(\ref{StrichartzInequality4})$ in $(\ref{EQ12})$, it follows that
\begin{equation}\label{EQ13}
\|\Phi(u) \|_{L^{p}([0,T];W^{\sigma, q})}  \lesssim \|u_{0} \|_{H^{s}} + \|G \|_{L^{1}([0, T]; H^{s}(M))}.
\end{equation}

Inserting \eqref{EQ13} in \eqref{EQ11.1}, yields 
\begin{equation}\label{EQ13a}
\|\Phi(u) \|_{L^{p}([0,T];L^{\infty}(M))}  \lesssim \|u_0\|_{H^{s}} + \|G \|_{L^{1}([0, T]; H^{s}(M))}.
\end{equation}

Finally, combining  $(\ref{EQ11})$ and $(\ref{EQ13a})$, we find
\begin{equation}\label{EQ14}
\|\Phi(u) \|_{X_{T}}  \lesssim \|u_{0} \|_{H^{s}} +
  \|G  \|_{L^{1}([0, T]; H^{s}(M))}.
\end{equation}

Now, we move to estimate the second term in the right hand side of $(\ref{EQ14})$. We start noting that
\begin{equation}\label{EQ16}
\begin{split}
  \|G \|_{L^{1}([0, T]; H^{s}(M))}  &\leq  \int_{0}^{T} \|G_{1}(\ell) \|_{H^{s}} d\ell +\int_{0}^{T} \|G_{2}(\ell) \|_{H^{s}} d\ell \\
& \leq   \int_{0}^{T} e^{- \frac{\ell}{\kappa}} \|u v_{0}\|_{H^{s}} d\ell  + \frac{1}{\kappa} \int_{0}^{T}\left\| u(\ell) \int_{0}^{\ell} e^{- \frac{(\ell - t')}{\kappa}} |u(t')|^{2} d t'\right\|_{H^{s}}d\ell \\
& =: I_{A} + \frac{1}{\kappa} I_{B} .
\end{split}
\end{equation}

Now, we will estimate each one of the terms in  $(\ref{EQ16})$ separately.
\begin{enumerate}
\item[$\bullet$] Using  Lemma $\ref{EQ1}$ and   Hölder's inequality, we get
\begin{equation}\label{EQ18}
\begin{split}
I_A
 & \lesssim  \int_{0}^{T} \|u \|_{L^{\infty}}  \|v_{0}\|_{H^{s}} d\ell +  \int_{0}^{T} \|u \|_{H^{s}}  \|v_{0}\|_{L^{\infty}} d\ell \\
& \lesssim   \|v_{0}\|_{H^{s}}  \|u \|_{L_{T}^{p}L_{x}^{\infty}}  T^{1 - \frac{1}{p}} +  \|u \|_{L_{T}^{\infty}H^{s}}    \|v_{0}\|_{L^{\infty}} T.
\end{split}
\end{equation}

Considering $0<T\leq 1$ and  using the definition of the  $\|\cdot\|_{X_{T}}$-norm in $(\ref{EQ18})$, we obtain
\begin{equation}\label{EQ19}
 I_{A}\lesssim 2 \|v_{0}\|_{H^{s} \cap L^{\infty}}  \|u \|_{X_{T}}  T^{\gamma_{p}},
\end{equation}
where  $\gamma_{p} := 1 - \frac{1}{p} > 0$.

\item[$\bullet$] Now, we estimate $I_{B}$. Using Lemma \ref{EQ21Star} with $f=u$, $g=u$ and $h=\bar{u}$, we obtain
\begin{equation}\label{EQ20}
\begin{split}
I_{B}& \lesssim \int_{0}^{T}\left\| u(\ell) \int_{0}^{} e^{- \frac{(\ell- t')}{\kappa}} |u(t')|^{2} d t' \right \|_{H^{s}} d\ell\\
&\lesssim   \|u\|_{X_{T}}^2 \int_{0}^{T}\|{u}(\ell)\|_{H^{s}} d\ell + 2  T^{\gamma_{p}}  \|u\|_{X_{T}}^2  \int_{0}^{T}\|{u}(\ell)\|_{L^{\infty}_{x}} d\ell.
\end{split}
\end{equation}
Using H\"olders' inequality, it follows  that
\begin{equation}\label{EQ26}
\begin{split}
 I_{B} &\lesssim  \|u\|_{X_{T}}^2 \|u\|_{L_T^{\infty}H^s} T+ 2  \|u\|_{X_{T}}^2 \|{u}\|_{L_{T}^pL_x^{\infty}} T^{2 \gamma_{p}}.
\end{split}
\end{equation}

Therefore, considering $0<T\leq 1$, we get
\begin{equation}\label{EQ27}
\frac{1}{\kappa}I_{B} \lesssim  \frac{3}{\kappa} \|u\|_{X_{T}}^{3}  T^{ \gamma_{p}}.
\end{equation}

\end{enumerate}

Thus, in view of estimates  $(\ref{EQ16})$, $(\ref{EQ19})$ and $(\ref{EQ27})$  we obtain from  $(\ref{EQ14})$ that
\begin{equation}\label{EQ28}
\begin{split}
\|\Phi(u)\|_{X_{T}}  & \lesssim \|u_{0}\|_{H^{s}} + 2\|v_{0}\|_{H^{s} \cap L^{\infty}}  \|u \|_{X_{T}}  T^{\gamma_{p}}
+ \frac{3}{\kappa}  \|u \|_{X_{T}}^3 T^{ \gamma_{p}}.
\end{split}
\end{equation}

With  estimate $(\ref{EQ28})$ at hand, we have the following consequences.

\noindent
{\bf (a)} $\Phi$ maps the ball $B_{T}^{R}$ onto itself, for  suitable values of $T, R >0$. In fact, from  $(\ref{EQ28})$, we get
\begin{equation}\label{EQ28a1}
\|\Phi(u) \|_{X_{T}}
\leq C \Big( \|u_{0} \|_{H^{s}} + \|v_{0}\|_{H^{s} \cap L^{\infty}}  \|u \|_{X_{T}}  T^{\gamma_{p}}
+  \frac{1}{\kappa} \|u\|_{X_{T}}^{3} T^{  \gamma_{p}}\Big).
\end{equation}

Let us choose  $R:= 2 C \|u_{0}\|_{H^{s}} > 0$ and consider $u \in  B_{R}^{T} \subset X_{T}$. With these considerations, 
\eqref{EQ28a1} yields  
\begin{equation}\label{EQ28a2.1}
\begin{split}
\|\Phi(u) \|_{X_{T}}
&\leq  \frac{R}{2} + C \left(\|v_{0}\|_{H^{s} \cap L^{\infty}}
+  \frac{1}{\kappa} R^{2} \right)   T^{\gamma_{p}} R.
\end{split}
\end{equation}
Now, we choose $T > 0$ in such a way that,
\begin{equation}\label{Choice-T}
C \left(\|v_{0}\|_{H^{s} \cap L^{\infty}}
+  \frac{1}{\kappa} R^{2} \right)  T^{\gamma_{p}} < \frac{1}{2},
\end{equation}
which means,
\begin{equation}\label{Choice-T2}
 T \simeq \left(\frac{1}{ \|v_{0}\|_{H^{s} \cap L^{\infty}}
+  \frac{1}{\kappa} \|u_0\|_{H^s}^{2} } \right)^{\frac{1}{\gamma_{p}}}.
\end{equation}
Using these choices in \eqref{EQ28a2.1} we conclude that $\Phi$ maps $B_{T}^R$ onto itself.

\noindent
{\bf (b)} $\Phi: B_{T}^{R} \rightarrow B_{T}^{R}$ is a contraction map. In fact, using relations $u v_{0} -   \widetilde{u} \widetilde{v_{0}}=u (v_{0} -   \widetilde{v_{0}})+   \tilde{v_{0}} (u -   \tilde{u})$  and  $|u|^{2} - |\tilde{u}|^{2} =  \overline{u} (u - \tilde{u}) + \tilde{u} \overline{(u - \tilde{u})}$ and an analogous procedure to obtain \eqref{EQ28}, we can easily get
\begin{equation}\label{Contra-1}
\begin{split}
\|\Phi(u) - \Phi(\tilde{u})\|_{X_{T}}
& \lesssim \| v_{0}\|_{ H^{s} \cap L^{\infty}}  \|u -  \tilde{u}\|_{X_{T}} T^{\gamma_{p}}\\
 &\quad + \frac{3}{\kappa} (\|u\|_{X_{T}} \|\tilde{u}\|_{X_{T}}
+   \|u\|_{X_{T}}^{2}
+   \|\widetilde{u}\|_{X_{T}}^{2})  \|u - \tilde{u}\|_{X_{T}}  T^{ \gamma_{p}}.
 \end{split}
 \end{equation}
Therefore, for   $u, \tilde{u} \in B_{T}^R$ we get
\begin{equation}\label{Contra-2}
\|\Phi(u) - \Phi(\tilde{u})\|_{X_{T}}\leq  C ( \| v_{0}\|_{ H^{s} \cap L^{\infty}}
 + \frac{1}{\kappa} R^{2})  \|u - \tilde{u}\|_{X_{T}}  T^{ \gamma_{p}}.
\end{equation}

Hence, for the choice of $T$ as in \eqref{Choice-T} (possibly smaller) we have  $ C ( \|v_{0}\|_{ H^{s} \cap L^{\infty}}
 + \frac{1}{\kappa} R^{2})   T^{ \gamma_{p}} < 1$, and consequently  conclude that  $\Phi: B_{T}^R \rightarrow B_{T}^R$ is a contraction. Thus, applying the Banach fixed-point theorem, we obtain that there exists  a unique
 $u \in B_{T} \subset X_{T}$ which is solution of the integral equation $(\ref{SDS1234})$.

\noindent
{\bf (c)} The flow generated  by the solution is Lipschitz in bounded subsets of $H^{s}$. In fact, consider $u , \tilde{u} \in X_{T}$ two solutions of the integral equation corresponding to the two initial data $u_{0}, \tilde{u_{0}} \in H^{s}(M)$ (noting that $v_{0}$ is the same for the both solutions). Thus, with a similar procedure used to obtain \eqref{Contra-1}, we get
\begin{equation}\label{EQ28a2}
\begin{split}
\|u- \tilde{u}\|_{X_{T}} & \leq C \|u_{0} -  \tilde{u_{0}}\|_{H^{s}} \\
& \quad{ }+  C \left(\| v_{0}\|_{ H^{s} \cap L^{\infty}}  + \frac{3}{\kappa} \Big( \|u\|_{X_{T}}^{2}  +  \|u\|_{X_{T}} \|\tilde{u}\|_{X_{T}} +  \|\widetilde{u}\|_{X_{T}}^{2} \Big) \right) \|u -   \tilde{u}\|_{X_{T}} T^{\gamma_{p}} .
\end{split}
\end{equation}

Recall that, in $(a)$ and $(b)$ that we had chosen  $T > 0$ and $R >0$ in such a way that $ C ( \|v_{0}\|_{ H^{s} \cap L^{\infty}}
 + \frac{1}{\kappa}R^{2})   T^{ \gamma_{p}} < 1$. Using these choices we deduce from 
 \eqref{EQ28a2} that, if  $u, \tilde{u} \in B_{T}^{R}$, then there exists
 $C > 0$ such that
\[
\|u -  \tilde{u} \|_{X_{T}} \leq C \|u_{0}- \tilde{u}_{0}\|_{ H^{s}}.
\]
Thus, we conclude that the solution is in fact Lipschitz on bounded subsets of $H^{s}(M)$.

\noindent
{\bf (d) }Persistence of regularity. It remains to verify that the property  $(ii)$ of the  Definition $\ref{LWPDef}$ holds true. Using  $(\ref{EQ14})$  for the solution $u=\Phi(u)$ with initial data  $u_{0} \in H^{r}(M)$ with  $r > s$, we have
\begin{equation}\label{EQ28a}
\|u(t)  \|_{H^{r}} \lesssim \|u_0\|_{H^r}+\int_{0}^{T}  \| G(t)\|_{H^{r}} dt.
\end{equation}

Now, observe that
\begin{equation}\label{EQ28aa}
\begin{split}
\int_{0}^{T}\|G(t)\|_{H^{r}} dt&\leq \int_{0}^{T} \|u(t)v_{0}(x)\|_{H^{r}} dt+ \int_{0}^{T} \left\|u(t) \int_{0}^{t} e^{-(t- t')/\kappa} |u(t')|^{2} dt' \right\|_{H^{r}} dt\\
& =: A_{1} + A_{2}. \\
\end{split}
\end{equation}

Using $(\ref{EQ18})$ with $s = r$, we get
\begin{equation}\label{EQ28b}
A_{1}  \lesssim \|v_{0}\|_{H^{r}} \|u\|_{L^{p}_{T} L^{\infty}_{x}} T^{\gamma_{p}} + \|u\|_{L^{\infty}_{T}(H^{r})}\|v_{0}\|_{L_{x}^{\infty}} T.
\end{equation}

On the other hand, using   $(\ref{EQ21-1})$ with $s = r$, $f,g=u$, $h=\bar{u}$ and $p > 2$, one obtains
\[
A_{2}\lesssim T \|u\|_{L^{\infty}_{T}(H^{r})} \|u\|_{L^{p}_{T}(L^{\infty}_{x})}^{2}  + 2 \|u\|_{L^{\infty}_{T}(H^{r})} \|u\|_{L^{p}_{T}(L^{\infty}_{x})} T^{\gamma_{p}}  \int_{0}^{T} \|u(t)\|_{L^{\infty}_{x}} dt.
\]

Next, using Hölder's inequality, in the variable $t$, we get
\[
A_{2}\lesssim  (T + 2 T^{2\gamma_{p}}) \|u\|_{L^{\infty}_{T}(H^{r})} \|u\|_{L^{p}_{T}(L^{\infty}_{x})}^{2}.
\]

As $p > 2$, we have $2 \gamma_{p} = 2 - \frac{2}{p}> 1$. Thus, if $T \leq 1$, we obtain
\begin{equation}\label{EQ28c}
A_{2}\lesssim  3 T \|u\|_{L^{\infty}_{T}(H^{r})} \|u\|_{L^{p}_{T}(L^{\infty}_{x})}^{2}.
\end{equation}

Combining  estimates \eqref{EQ28a}, \eqref{EQ28aa}, \eqref{EQ28b} and \eqref{EQ28c}, we arrive at 
\begin{equation}\label{EQ28d}
\|u(t)\|_{H^{r}} \lesssim \|u_{0}\|_{H^{r}} +  \|v_{0}\|_{H^{r}} \|u\|_{L^{p}_{T} L^{\infty}_{x}} T^{\gamma_{p}} + T \|u\|_{L^{\infty}_{T}(H^{r})} \|v_{0}\|_{L_{x}^{\infty}} +   3 T \|u\|_{L^{\infty}_{T}(H^{r})} \|u\|_{L^{p}_{T}(L^{\infty}_{x})}^{2}.
\end{equation}
Denote by $T_{r}$ the time of local existence and by  $T^{\ast}_{r}$ the maximal time of existence for the solution  $\tilde{u}$ corresponding to
the initial data $u_{0} \in H^{r}$. We have that   $T^{\ast}_{r}$ satisfies the blow-up alternative with respect to the norm
$\|.\|_{H^{r}}$. As $H^{r} \hookrightarrow H^{s}$, we have to make a distinction between the solutions in these spaces, because
$\tilde{u}$  can be taken as a solution in $H^{s}$ too. Denote by $T_{s}$  the time of local existence of solution in $H^{s}$, and by $u$ the solution in this space. By uniqueness, we have  $u = \tilde{u}$ in $[0, T_{r}^{\ast})=:J_r^*$. We want to show that $T_{s}^{\ast} = T_{r}^{\ast}$,
where $T_{s}^{\ast}$ denotes the maximal time existence of the solution  $u$. Clearly, we have   $T_{s}^{\ast} \geq T_{r}^{\ast}$. Suppose $T_{s}^{\ast} > T_{r}^{\ast}$,
and consider $0 < \varepsilon <L\ll 1$ in such way that $I_{\varepsilon} := [T^{\ast}_{r} - L,T^{\ast}_{r} - \varepsilon ]$
satisfies $I_{\varepsilon} \subset J_r^*$  and
\[
|I_{\varepsilon}| = T^{\ast}_{r} - \varepsilon  - (T^{\ast}_{r} - L) = L - \varepsilon > 0.
\]

Using the estimate $(\ref{EQ28d})$ on the interval $I_{\varepsilon}$, we obtain
\begin{equation}\label{EQ28e}
\begin{split}
\|\tilde{u}(t)\|_{H^{r}} & \lesssim \|u(T^{\ast}_{r} - L)\|_{H^{r}} +  \|v_{0}\|_{H^{r}} \|\tilde{u}\|_{L^{p}(I_{\varepsilon}, L^{\infty}_{x})} L^{\gamma_{p}}\\
&\quad+ L \cdot \|\tilde{u}\|_{L^{\infty} ( I_{\varepsilon},H^{r})} \Big(\|v_{0}\|_{L_{x}^{\infty}} + 3  \|\tilde{u}\|_{L^{p}(I_{\varepsilon},L^{\infty}_{x})}^{2} \Big) .
\end{split}
\end{equation}

Notice that, $\tilde{u} = u$ in  $J_{r}^{\ast}$ and  $u$  exists in  $[0, T_{s}^{\ast})\supset J_{r}^{\ast}$. This shows that
\[
\|\tilde{u}\|_{L^{p}([0, T^{\ast}_{r}), L^{\infty}_{x})} = \|u\|_{L^{p}(\overline{J_{r}^{\ast}}, L^{\infty}_{x})} < \infty.
\]

It follows from $(\ref{EQ28e})$ that
\begin{equation}\label{EQ28f}
\begin{split}
\|\tilde{u}\|_{L^{\infty}(I_{\varepsilon},H^{r})} & \lesssim \|u(T^{\ast}_{r} - L)\|_{H^{r}} +  \|v_{0}\|_{H^{r}} \|u\|_{L^{p}(J^{\ast}_{r}, L^{\infty}_{x})} L^{\gamma_{p}}\\
 &\quad{} + L \cdot \|\tilde{u}\|_{L^{\infty}( I_{\varepsilon},H^{r})}  (\|v_{0}\|_{L_{x}^{\infty}} +   3  \|u\|_{L^{p}(\overline{J_{r}^{\ast}},L^{\infty}_{x})}^{2}) .
\end{split}
\end{equation}

Thus, if we choose $L$ very small in such that $ C L  (\|v_{0}\|_{L_{x}^{\infty}} + 3 \|u\|_{L^{p}(\overline{J_{r}^{\ast}},L^{\infty}_{x})}^{2}) < \frac{1}{2}$, it follows that
\begin{equation}\label{EQ28g}
\|\tilde{u}\|_{L^{\infty}(I_{\varepsilon},H^{r})} \leq  2C \|u(T^{\ast}_{r} - L)\|_{H^{r}} +  2 C\|v_{0}\|_{H^{r}} \|u\|_{L^{p}(J^{\ast}_{r}, L^{\infty}_{x})}.
\end{equation}
Hence, $ \lim_{t \rightarrow T_{r}^{\ast}} \|\tilde{u}(t)\|_{H^{r}} < \infty$, and consequently, using  the blow-up alternative we should have
 $T_{r}^{\ast} = \infty$, what is a contradiction, because we supposed $T_{r}^{\ast} < T_{s}^{\ast}$ (that is, $T_{r}^{\ast} < \infty$). Thus,
we must have $T^{\ast}_{r} = T^{\ast}_{s}$ as desired.
\end{proof}

\begin{obs}
For $s > \frac{d}{2}$, using the Sobolev embedding $H^{s}(M) \hookrightarrow L^{\infty}(M)$,
it is possible to prove that, for every $(u_{0}, v_{0}) \in H^{s}(M) \times H^{s}(M)$, there exist a time  $T = T (\kappa,\|v_{0}\|_{H^{s}}, \|u_{0}\|_{H^{s}})$ and 
a unique solution $u \in X_{T}= C([0, T]; H^{s}(M))$ of \eqref{IntegroDifferentialFormulation}. 
Also, it can be shown that the time of existence $T$  obeys the relation
\[
T \simeq \Big(\frac{1}{\|v_{0}\|_{H^{s}} \sqrt{\kappa} + \|u_{0}\|_{H^{s}}^{2}} \Big)^{\frac{1}{2}}.
\]
Thus, in the limit when $\kappa \longrightarrow 0^{+}$, we have that $T \simeq \frac{1}{\|u_{0}\|_{H^{s}}}$, which is a contrast 
with the time of local existence in the case $s < \frac{d}{2}$ (see \eqref{Choice-T2} above). This allows one to think about  issues related to the  convergence of solutions of the SD system to solutions of the NLS equation when the parameter $\kappa $ tends to zero considering $s>\frac{d}2$.

\end{obs}

Now, we move prove the persistence property for $v$.

\begin{proof}[Proof of Corollary \ref{v(t)Persistence}]
Suppose that $v_{0} \in H^{s} \cap L^{\infty}$. According to the Theorem \ref{TeorIntro5}, one has
\[ 
u \in C([0, T], H^{s}) \cap L^{p}([0, T], L^{\infty}). 
\] 

First, we  show that $v \in C([0, T], H^{s})$. Consider $t_{0} \in (0, T)$ and $0 < \varepsilon \ll 1$. Now, using the expression for $v$ given by \eqref{ODEvsolution}, we get 
\begin{equation*}
\begin{split}
 \|v(t_{0} + \varepsilon) - v(t_{0})\|_{H^{s}} & \leq e^{-\frac{t_{0}}{\kappa}} |e^{-\frac{\varepsilon}{\kappa}} - 1|  \|v_{0}\|_{H^{s}} + \frac{C}{\kappa} \int_{t_{0}}^{t_{0} + \varepsilon} e^{ -\frac{(t_{0} - \ell)}{\kappa}} \|u(\ell)\|_{L^{\infty}} \|u(\ell)\|_{H^{s}} d \ell\\
 &\leq |e^{-\frac{\varepsilon}{\kappa}} - 1|  \|v_{0}\|_{H^{s}} + \frac{C}{\kappa}  \|u\|_{L^{\infty}([t_{0}, t_{0}+ \varepsilon ],H^{s})}   \int_{t_{0}}^{t_{0} + \varepsilon} e^{ -\frac{(t_{0} - \ell)}{\kappa}} \|u(\ell)\|_{L^{\infty}} d \ell.
\end{split}
\end{equation*}
Let $p > 2$ and $p'$  be the conjugate exponent. Using  Hölder's inequality, we get
\begin{equation*}
\begin{split}
 \|v(t_{0} + \varepsilon) - v(t_{0})\|_{H^{s}}& \leq  |e^{-\frac{\varepsilon}{\kappa}} - 1|  \|v_{0}\|_{H^{s}} + \frac{C}{\kappa}  \|u\|_{L^{\infty}([t_{0}, t_{0}+ \varepsilon ],H^{s})} \|u\|_{L^{p}([t_{0}, t_{0} + \varepsilon],L^{\infty})}  \Big(\int_{t_{0}}^{t_{0} + \varepsilon} e^{ -\frac{(t_{0} - \ell)}{\kappa} p'} d \ell\Big)^{\frac{1}{p'}}\\
 & \leq  |e^{-\frac{\varepsilon}{\kappa}} - 1|  \|v_{0}\|_{H^{s}  \cap L^{\infty}} + \frac{C}{\kappa}  \|u\|_{L^{\infty}([0, T ],H^{s})} \|u\|_{L^{p}([0, T],L^{\infty})}  \Big( \frac{\kappa}{p'}\Big)^{\frac{1}{p'}} 
 \Big(e^{\frac{\varepsilon p'}{\kappa}} - 1\Big)^{\frac{1}{p'}}.
\end{split}
\end{equation*}
Passing to the limit when $\varepsilon \rightarrow 0^{+}$ we obtain the continuity from the right. In an analogous 
way, we obtain the continuity from the left. This leads to the conclusion that
$v \in C([0, T], H^{s})$ . 

Second, to show that $v \in L^{p}([0, T], L^{\infty}(M))$ we take the $L^{\infty}$- norm of $v$ in \eqref{ODEvsolution}
and use Hölder's inequality to obtain 
\[
 \|v(t)\|_{L^{\infty}} \leq e^{\frac{-t}{\kappa}}  \|v_{0}\|_{L^{\infty}} + \kappa^{-\frac{2}{p}}  \Big( \frac{p-2}{p}\Big)^{\frac{p-2}{p}} \|u\|_{L^{p}([0, T],L^{\infty})}^{2}.
\]
Next, taking the $L^{p}_{T}$- norm, one obtains 
\[
 \|v\|_{L^{p}([0, T],L^{\infty})} \leq  \|e^{\frac{-t}{\kappa}}\|_{L^{p}([0, T])}   \|v_{0}\|_{L^{\infty}} + \kappa^{-\frac{2}{p}}  \Big( \frac{p-2}{p}\Big)^{\frac{p-2}{p}} \|u\|_{L^{p}([0, T],L^{\infty})}^{2} T^{\frac{1}{p}}.
\]

Therefore, 
\[
 \|v\|_{L^{p}([0, T],L^{\infty})} \leq  (1 - e^{-\frac{Tp}{\kappa}})^{\frac{1}{p}} \Big( \frac{\kappa}{p}\Big)^{\frac{1}{p}} \|v_{0}\|_{H^{s} \cap L^{\infty}} + \kappa^{-\frac{2}{p}}  \Big( \frac{p-2}{p}\Big)^{\frac{p-2}{p}} \|u\|_{L^{p}([0, T],L^{\infty})}^{2} T^{\frac{1}{p}}.
\]

This shows  that $v \in L^{p}([0, T], L^{\infty})$ for $v_{0}  \in H^{s} \cap L^{\infty}$, and 
 completes the proof of the Corollary. 

\end{proof}

\begin{obs}\label{Vpersistence}

For given $v_{0} \in L^{\infty}$, its evolution $v(t)$  remains in 
$L^{\infty}$ for all $0 \leq t \leq T$. In fact, using $u\in X_T$, we have from  \eqref{ODEvsolution}
\begin{equation}\label{VLinftyEstimate}
\begin{split}
 \|v(t)\|_{L^{\infty}(M)} &= e^{\frac{-t}{\kappa}} \|v_{0}\|_{L^{\infty}(M)} + \frac{1}{\kappa} \int_{0}^{t} e^{\frac{-(t - t')}{\kappa}} \|u(t')\|_{L^{\infty}}^{2} dt'\\
&\leq  \|v_{0}\|_{L^{\infty}(M)} + \frac{1}{\kappa} \int_{0}^{t}  \|u(t')\|_{L^{\infty}}^{2} dt' \\
&\leq  \|v_{0}\|_{L^{\infty}(M)} + \frac{1}{\kappa} \left(\int_{0}^{t}  \|u(t')\|_{L^{\infty}}^{p}  dt'\right)^{\frac{2}{p}} \left(\int_{0}^{t} dt'\right)^{1 - \frac{2}{p}} \\
&=  \|v_{0}\|_{L^{\infty}(M)} + \frac{1}{\kappa}  \|u\|_{L^{p}([0, t];L^{\infty})}^{2} \cdot  t^{1 - \frac{2}{p}} <\infty.
\end{split}
\end{equation}

In the case in that $d = 2$ we have $H^{1 + } \hookrightarrow L^{\infty}$. Thus, if $v_{0} \in H^{1+}$ 
the estimate \eqref{VLinftyEstimate} holds true if $u \in L^{p}([0, T], L^{\infty})$. 
\end{obs}

\section{Global well-posedness}\label{sec-4}
In this section we prove the global well-posedness result for the SD-system considering $d=2$.
\subsection{Derivation of an {\em a priori} estimate} 
 Recall that
the SD system  does not possess energy-conservation law. However, we can obtain a good relation involving the gradient term.
Using such a relation, we deduce an \textit{a priori} estimate for the $H^{1}$-norm of the  solution $u$, which in turn will be used to prove the global well-posedness result. 

The following version of the Gagliardo-Nirenberg inequality on compact Riemannian manifolds will be very useful in our argument.
\begin{prop}\label{GNM}
Let $(M, g)$ be a compact Riemannian manifold  of dimension $d \geq2$ and  $1 < p \leq 2$.  If
$1 \leq q < r < p^{\ast} = \frac{dp}{d-p}$ and $\theta := \theta(p, q, r) = \frac{d p (r - q)}{r ( q (p - d) + d p)} \in (0, 1]$, then

\begin{equation}\label{EQ34}
\|u\|_{L^{r}(M)}^{\frac{p}{\theta}} \leq  \left( A_{opt} \| \nabla_{g} u\|_{L^{p}(M)}^{p} + B \|u\|_{L^{p}(M)}^{p} \right) \|u\|_{L^{q}(M)}^{\frac{p (1 - \theta)}{\theta}}
\end{equation}
where
\begin{equation}\label{EQ35}
A_{opt} = \inf \Big\{ A \in \mathbb{R}: \mbox{ there is } B \in \mathbb{R} \mbox{ such that } (\ref{EQ34}) \mbox{ is valid }  \Big\}.
\end{equation}
\end{prop}
\begin{dem}
See  \cite{CECCONMONTENEGRO}, pg 854.
\end{dem}

Now, we move to derive a relation involving gradient of the solution $u$ to the IVP \eqref{SDE12} will  be useful to obtain an \textit{a priori} estimate. It
 can be seen as a
modification of equation $(4.2)$ in \cite{BIDEGARAY2} or a generalization
of Proposition 5, pg.710 in \cite{ARBIETOMATHEUS} which we extend to the SD system in the context of  Riemannian manifolds.
\begin{prop}\label{PropertieOfSolution}
Let $u$ be a sufficiently smooth solution to the SD system  \eqref{SDE12}. Then,
\begin{equation}\label{TimeDerivative}
    \frac{d}{dt}\left(\int_{M} |\nabla u(t)|_{g}^{2} dg + \int_{M} |u(t)|^{2} v(t) dg\right) = \frac{1}{\kappa} \left( - \int_{M} |u(t)|^{2} v(t) dg + \lambda \int_{M} |u(t)|^{4} dg\right).
\end{equation}
\end{prop}

\begin{dem}We will give a proof for $d=2$. The proof holds likewise for $d \geq 3$ by using a suitable choice of $(u_0, v_0)$ 
in order to obtain a sufficiently smooth solution $u$. Let $u(t, \cdot)$ be the $H^{2}$- solution for initial data $u_0, v_0 \in H^{2}(M)$.



Next, define 
\[
E(t):= \int_{M} |\nabla_{g} u(t,\cdot)|_{g}^{2} dg,
\]
so that, 
\[
\frac{E(t + \varepsilon) -E(t) }{\varepsilon} = \int_{M} \frac{ |\nabla_{g} u(t + \varepsilon,\cdot)|_{g}^{2} -  |\nabla_{g} u(t,\cdot)|_{g}^{2}}{\varepsilon} dg.
\]

Using integration by parts formula and taking into account that 
$\partial M = \emptyset$, we obtain 
\[
- \Big(\frac{E(t + \varepsilon) -E(t) }{\varepsilon} \Big) = \int_{M} \frac{u(t + \varepsilon) - u(t)}{\varepsilon} \Delta \bar{u}(t + \varepsilon) dg+  \int_{M} \frac{\bar{u}(t + \varepsilon) - \bar{u}(t)}{\varepsilon} \Delta u(t + \varepsilon) dg.
\]

Now, define 
\[
\tilde{E}(t) = -\int_{M} (\partial_{t} u(t) \Delta \bar{u}(t) + \Delta u(t) \partial_{t} \bar{u}(t)) dg = -2 \int_{M}\Re \Big(\partial_{t} u(t) \Delta \bar{u}(t)\Big) dg.
 \]
Thus, 
\begin{equation}\label{GL-10}
\begin{split}
\Big| \tilde{E}(t) - \frac{E(t + \varepsilon) -E(t) }{\varepsilon}  \Big|& = \Big|  \frac{E(t + \varepsilon) -E(t) }{\varepsilon} - \tilde{E}(t) \Big| \\
&
\leq \Big| \int_{M} \partial_{t}u(t) \Delta \bar{u}(t)  -  \frac{u(t + \varepsilon) - u(t)}{\varepsilon} \Delta \tilde{u}(t + \varepsilon) dg  \Big| \\
&\quad
+ \Big| \int_{M} \partial_{t}\bar{u}(t) \Delta u(t)  -  \frac{\bar{u}(t + \varepsilon) - \bar{u}(t)}{\varepsilon} \Delta u(t) dg  \Big|\\
& =: E_{1} + E_{2}.
\end{split}
\end{equation}

On the one hand, 
\begin{equation}\label{GL-11}
\begin{split}
E_{1} &\leq  \Big| \int_{M} \partial_{t}u(t) (\Delta \bar{u}(t)  - \Delta \bar{u}(t + \varepsilon)) \Big| +  \Big| \int_{M} \Big( \partial_{t} u(t) - \frac{u(t + \varepsilon) - u(t)}{\varepsilon}\Big) \Delta \tilde{u}(t + \varepsilon) dg  \Big|\\
&
\leq \|\partial_{t} u\|_{L^{2}}  \|u(t + \varepsilon) - u(t)\|_{H^{2}} +  \|u(t + \varepsilon)\|_{H^{2}}  \Big\| \partial_{t} u(t) - \frac{u(t + \varepsilon) - u(t)}{\varepsilon} \Big\|_{L^{2}}.
\end{split}
\end{equation}
Using the fact that $u \in C([0, T], H^{2}(M)) \cap C^{1}([0, T], L^{2}(M))$, we conclude that $E_{1}$ converges to zero
when $\varepsilon \rightarrow 0$. 

On the other hand, 
\[
E_{2} \leq  \|u(t)\|_{H^{2}}  \Big\| \partial_{t} u(t) - \frac{u(t + \varepsilon) - u(t)}{\varepsilon} \Big\|_{L^{2}}.
\]
and follows that  $E_{2}$ converges to zero
when $\varepsilon \rightarrow 0$. This shows that, 

\begin{equation}\label{FirstTimeEquality}
\begin{split}
\frac{1}{2}\frac{d}{dt} E(t) &= -   \int_{M}\Re \Big(\partial_{t} u(t) \Delta \bar{u}(t)\Big) dg.\\
\end{split}
\end{equation}

Now, write
 $u = \alpha + i \beta$. As $u$ satisfies  the first equation in \eqref{SDE12}, comparing the real and imaginary parts, we obtain
\begin{equation}\label{RealAndImaginaryDerivatives}
\begin{cases}
 \partial_{t} \alpha = - \Delta \beta + \beta v,  \\
 \partial_{t} \beta = \Delta \alpha - \alpha v.
\end{cases}
\end{equation}
Thus, from \eqref{FirstTimeEquality} we obtain that
\begin{equation}\label{FirstTimeEquality2}
\begin{split}
\frac{1}{2}\frac{d}{dt} E(t) &=-  \int_{M} (\partial_{t} \alpha \Delta \alpha  + \partial_{t} \beta \Delta \beta)dg \\
  &= -   \int_{M} v (\beta \Delta \alpha - \alpha \Delta \beta)dg, 
\end{split}
\end{equation}

where we used the relation \eqref{RealAndImaginaryDerivatives}. 

Now, using a similar reasoning used  to compute $\frac{d}{dt}E(t)$, one gets 
\begin{equation}\label{RegraDaCadeia}
\frac{1}{2} \frac{d}{dt} \int_{M} |u(t)|^{2} v(t) dg = \frac{1}{2} \int_{M} v(t) \partial_{t} (|u(t)|^{2}) dg+ \partial_{t} v |u(t)|^{2} dg.
\end{equation}
 Using $(\ref{RealAndImaginaryDerivatives})$ again, we obtain
\begin{equation}\label{1990}
\partial_{t} (|u(t)|^{2}) = 2 (\beta \Delta \alpha - \alpha \Delta \beta).
\end{equation}

Moreover, from the second equation in  $(\ref{SDE12})$, it follows that
\begin{equation}\label{1991}
\partial_{t} v = \frac{1}{\kappa} (-v + \lambda |u|^{2}).
\end{equation}

 Using  $(\ref{1990})$ and  $(\ref{1991})$ in $(\ref{RegraDaCadeia})$, we obtain
\begin{equation}\label{SecondTimeEquality}
\frac{1}{2} \frac{d}{dt} \int_{M} |u(t)|^{2} v(t) dg = \int_{M} v(\beta \Delta \alpha - \alpha \Delta \beta) dg- \frac{1}{2 \kappa} \int_{M} |u(t)|^{2} v(t)dg + \frac{\lambda}{2 \kappa} \int_{M}|u(t)|^{4} dg.
\end{equation}

Finally, adding the identities  $(\ref{FirstTimeEquality})$ and $(\ref{SecondTimeEquality})$ we find the desired equality
\[
 \frac{d}{dt}\left(\int_{M} |\nabla u(t)|_{g}^{2}dg + \int_{M} |u(t)|^{2} v(t) dg \right) = - \frac{1}{\kappa} \left(\int_{M} |u(t)|^{2} v(t) dg - \lambda  \int_{M} |u(t)|^{4} dg\right).
\]
\end{dem}

Now we will justify the passing to the limit in order to obtain the estimate 
\eqref{TimeDerivative} for $H^{1}$-solutions in dimension 2. Let $u_{0} \in H^{1}(M)$ and 
$v_{0} \in H^{1+}(M)$.  By Theorem \ref{TeorIntro5}, there exists a solution 
$u$ of the integro-differential equation \eqref{IntegroDifferentialFormulation},
such that $u \in C([0, T]; H^{1}(M)) \cap L^{p}([0, T], L^{\infty})$. 
By density, there are  sequences  $(u_{0}^{n})_{n} \subset H^{2}(M)$ such that $u_{0}^{n} \longrightarrow u_{0}$ in $H^{1}(M)$
and   $(v_{0}^{n})_{n} \subset H^{2}(M)$ such that $v_{0}^{n} \longrightarrow v_{0}$ in $H^{1+}(M)$. By persistence of regularity
(see item $(ii)$ Theorem \ref{TeorIntro5})
the $H^{2}$-solutions $(u^{n})_{n}$ exist in the time interval $[0, T]$ as well. Moreover, as a consequence 
of the continuous dependence (see item $(iii)$ Theorem \ref{TeorIntro5}), we have that 
$u^{n}(t) \longrightarrow u(t) $ in $H^{1}$ whenever  $ t \in [0, T]$. Integrating the identity \eqref{TimeDerivative}, we see that the sequence $(u^{n}(t))_{n}$ satisfies 
\begin{equation}\label{UnIdentity}
\begin{split}
 \int_{M} |\nabla u^{n}(t)|^{2} dg + \int_{M} |u^{n}(t)|^{2} v^{n}(t) dg & =\int_{M} |\nabla u^{n}_{0}|^{2} dg + \int_{M} |u^{n}_{0}|^{2} v^{n}_{0} dg +\\
  & \qquad+ \frac{1}{\kappa} \int_{0}^{t} \left(- \int_{M} |u^{n}(t')|^{2} v^{n}(t') dg + \lambda\int_{M} |u^{n}(t')|^{4} dg \right) dt'.\\
\end{split}
\end{equation}

Considering the identity \eqref{UnIdentity}, let 
\[
J_{1}^{n} := \left| \int_{M} |\nabla u(t)|^{2} dg -  \int_{M} |\nabla u^{n}(t)|^{2} dg\right|.
\]
Then, using \eqref{MassConservation} we can write
\[
J_{1}^{n} = \left|\|u^{n}(t)\|_{H^{1}}^{2} - \|u(t)\|_{H^{1}}^{2} \right| =  \left(\|u^{n}(t)\|_{H^{1}} + \|u(t)\|_{H^{1}} \right)  \Big|\|u^{n}(t)\|_{H^{1}} - \|u(t)\|_{H^{1}} \Big|.
\]
Using triangle inequality, we obtain
\[
J_{1}^{n} \leq  \left(\|u^{n}(t)\|_{H^{1}} + \|u(t)\|_{H^{1}} \right) \|u^{n}(t)- u(t)\|_{H^{1}} .
\]

By the continuous dependence of the solution on  the initial data, we have
\[
\|u^{n}(t)- u(t)\|_{H^{1}} \leq C \|u^{n}_{0}- u_{0}\|_{H^{1}}.
\]

Thus,
\begin{equation*}
\begin{split}
J_{1}^{n} &\leq  C \left(\|u^{n}(t)\|_{H^{1}} + \|u(t)\|_{H^{1}} \right) \|u^{n}_{0}- u_{0}\|_{H^{1}}\\
&\leq   C \left(\|u\|_{L^{\infty}([0,T];H^{1})} + \|u_{0}^{n} - u_{0}\|_{H^{1}} \right) \|u^{n}_{0}- u_{0}\|_{H^{1}},
\end{split}
\end{equation*}
and consequently  $J_{1}^{n} \longrightarrow  0$ when $n\longrightarrow \infty$, whenever $t \in [0, T]$.

Now, denote
\[
J_{2}^{n} := \left| \int_{M} |u^{n}(t)|^{2} v^{n}(t) dg  -  \int_{M}|u(t)|^{2} v(t) dg\right|.
\]
Using triangle inequality, it follows that
\begin{equation*}
\begin{split}
J_{2}^{n} &\leq \left| \int_{M} |u^{n}(t)|^{2} \Big(v^{n}(t) - v(t)\Big) dg\right| + \left| \int_{M} \Big(|u^{n}(t)|^{2} -|u(t)|^{2}\Big) v(t) dg\right|\\
&\leq \|u^{n}(t)\|_{L^{4}}^2 \|v^{n}(t) - v(t)\|_{L^{2}} + \|v(t)\|_{L^{\infty}} \Big| \|u_{0}^{n}\|_{L^{2}}^{2} - \|u_{0}\|_{L^{2}}^{2}\Big|.
\end{split}
\end{equation*}
Now, observe that
\begin{equation*}
\begin{split}
\|v^{n}(t) - v(t)\|_{L^{2}} & \leq e^{-t/\kappa} \|v^{n}_{0} - v_{0}\|_{L^{2}} + \int_{0}^{t} e^{\frac{-(t- \ell)}{\kappa}} \Big\| |u^{n}(\ell)|^{2} - |u(\ell)|^{2} \Big\|_{L^{2}} d \ell\\
&\leq e^{-t/\kappa} \|v^{n}_{0} - v_{0}\|_{L^{2}} + \int_{0}^{t} e^{\frac{-(t- \ell)}{\kappa}} \Big( \|u^{n}(\ell)\|_{L^{4}} + \|u(\ell)\|_{L^{4}} \Big)  \| u^{n}(\ell) - u(\ell)\|_{L^{4}} d \ell.
\end{split}
\end{equation*}

Recall from \eqref{VLinftyEstimate}, that
\[
\|v(t)\|_{L^{\infty}} \leq \|v_{0}\|_{L^{\infty}} + \frac{1}{\kappa} T^{1 - \frac{2}{p}}\|u\|_{L^{p}([0, T]; L^{\infty})}  < \infty .
\]
Thus, using the embedding $H^{1}(M^{2}) \subset L^{4}(M^{2})$
we conclude that $J^{n}_{2} \longrightarrow 0$ when $n \longrightarrow +\infty$.

The convergence for the first two terms in the RHS of \eqref{UnIdentity} is similar to $J^n_1$ and $J^n_2$ respectively, so we omit the details. 

Next, consider
\[
J_{3}^{n} : =  \left|\int_{0}^{t} \int_{M} |u^{n}(t')|^{2} v^{n}(t') dg dt'  -  \int_{0}^{t} \int_{M}|u(t')|^{2} v(t') dg dt'\right|.
\]
Using triangle inequality we can write
\[
J_{3}^{n} \leq  \int_{0}^{t} \left| J^{n}_{2}(t')\right| dt'.
\]
By  uniform convergence of the term $J^{n}_{2}$ in $[0, T]$ we conclude that 
$J^{n}_{3} \longrightarrow 0$ when $n \longrightarrow +\infty$.

Finally, we define
\[
J_{4}^{n} := \Big| \int_{0}^{t} \int_{M} |u^{n}(t')|^{4} dg dt' -  \int_{0}^{t} \int_{M} |u(t')|^{4} dg dt'\Big|.
\]

Thus, 
\begin{equation}
\begin{split}
J_{4}^{n} &\leq  \int_{0}^{t}\Big| \|u^{n}(t')\|_{L^{4}}^{4} - \|u(t')\|_{L^{4}}^{4}\Big| dt'\\
 &\leq  \int_{0}^{t} (\|u^{n}(t')\|_{L^{4}}^{2} + \|u(t')\|_{L^{4}}^{2})\Big| \|u^{n}(t')\|_{L^{4}}^{2} - \|u(t')\|_{L^{4}}^{2}\Big| dt'\\
  & \leq \Big( \sum_{ \ell = 0}^{3}\|u^{n}\|_{L^{\infty}([0, T], L^{4}(M))}^{3-\ell}  \|u\|_{L^{\infty}([0, T], L^{4}(M))}^{\ell}\Big)  \int_{0}^{t}\|u^{n}(t')- u(t')\|_{L^{4}}dt'\\
   & \leq  T\Big( \sum_{ \ell = 0}^{3}\|u^{n}\|_{L^{\infty}([0, T], L^{4}(M))}^{3-\ell}  \|u\|_{L^{\infty}([0, T], L^{4}(M))}^{\ell}\Big)\|u^{n}- u\|_{L^{\infty}([0, T], H^{1}(M))}  \\
   &\leq C T \Big(\sum_{ \ell = 0}^{3}\|u^{n}\|_{L^{\infty}([0, T], L^{4}(M))}^{3-\ell}  \|u\|_{L^{\infty}([0, T], L^{4}(M))}^{\ell}\Big) \|u_{0}^{n}- u_{0}\|_{H^{1}(M)} .\\
\end{split}
\end{equation}

Hence, $J_{4}^{n} \longrightarrow 0$ when $n \longrightarrow +\infty$ and completes the proof of \eqref{TimeDerivative} for $H^1$-solution.\\

Now, we derive an  \textit{a priori} estimate which is fundamental to prove the global well-posedness result.

\begin{prop}\label{PropAPrioriEstimate} If the $H^{1}$-solution $u$ of \eqref{IntegroDifferentialFormulation} exists in an interval 
$[0, T]$ satisfying  $$0 < T\leq \min \Big\{ 1 ,\frac{\kappa}{  12 C_{A_{opt}, B}\|u_{0}\|_{L^{2}}^{2}} \Big\}$$ 
then for all $t \in(0,T) $, we have
\begin{equation}\label{PrioriEstimate}
\|u(t)\|^{2}_{H^{1}} \leq 2 \|u_0\|_{H^{1}}^{2} +  18 C_{A_{opt}, B} \|u_{0}\|_{L^{2}}^{2}\|v_{0}\|_{L^{2}}^{2},
\end{equation}
 where $C_{A_{opt}, B} := \max\{ A_{opt}, B\}$.
\end{prop}

\begin{proof} Integrating  $(\ref{TimeDerivative})$ in $(0,t) \subset [0, T]$ , we get
\begin{equation}\label{EQ29}
\begin{split}
 \int_{M} |\nabla u(t)|^{2} dg + \int_{M} |u(t)|^{2} v(t) dg & =\int_{M} |\nabla u_{0}|^{2} dg + \int_{M} |u_{0}|^{2} v_{0} dg +\\
  & \qquad+ \frac{1}{\kappa} \int_{0}^{t} \left(- \int_{M} |u(t')|^{2} v(t') dg + \lambda\int_{M} |u(t')|^{4} dg \right) dt'.\\
\end{split}
\end{equation}

Using  $L^{2}$-norm conservation \eqref{MassConservation}, it follows from $(\ref{EQ29})$ that
\begin{equation}\label{EQ30}
\begin{split}
\|u(t)\|_{H^{1}}^{2}& = \|u_{0}\|_{H^{1}}^{2}+ \int_{M} |u_{0}|^{2} v_{0} dg - \int_{M} |u(t)|^{2} v(t) dg\\
  & \qquad - \frac{1}{\kappa} \int_{0}^{t}  \int_{M} |u(t')|^{2} v(t') dg dt' + \frac{\lambda}{\kappa} \int_{0}^{t} \int_{M} |u(t')|^{4} dg  dt'.\\
\end{split}
\end{equation}
Thus,
\begin{equation}\label{EQ31}
\begin{split}
\|u(t)\|_{H^{1}}^{2}&\leq  \|u_{0}\|_{H^{1}}^{2} + \int_{M} |u(t)|^{2} |v(t)| dg + \int_{M} |u_{0}|^{2} |v_{0}| dg +\\
  & \qquad + \int_{0}^{t}  \int_{M} |u(t')|^{2} |v(t')| dg dt' +  \frac{1}{\kappa} \int_{0}^{t} \int_{M} |u(t')|^{4} dg  dt'\\
  & =:\|u_{0}\|_{H^{1}}^{2} + I_{1} + I_{2} + I_{3} + I_{4} .
\end{split}
\end{equation}

In what follows we estimate each term $I_{j}$ ($j$ = 1,\ldots,4) in $(\ref{EQ31})$.

\begin{enumerate}
\item[$\bullet$]
Using Cauchy-Schwarz inequality, we have
\begin{equation}\label{EQ32}
I_{1} \leq \|u(t)\|_{L^{4}}^{2} \|v(t)\|_{L^{2}}.
\end{equation}

Using  Minkowski's inequality  and the Cauchy-Schwarz inequality in \eqref{ODEvsolution}, we conclude that
\begin{equation}\label{EQ33}
\begin{split}
\|v(t)\|_{L^{2}} &\leq  \|v_{0}\|_{L^{2}} + \frac{1}{\kappa}\int_{0}^{t}\| |u(t')|^{2}\|_{L^{2}} dt'\\
  & \leq\|v_{0}\|_{L^{2}} + \frac{1}{\kappa} \int_{0}^{T}\| u(t')\|_{L^{4}}^{2} dt' \\
  & \leq\|v_{0}\|_{L^{2}} + \frac{1}{\kappa} \left(\int_{0}^{T}\| u(t')\|_{L^{4}}^{4} dt'\right)^{1/2} T^{1/2}. \\
\end{split}
\end{equation}
This shows that it is necessary to introduce a suitable bound for the $L^{4}$ norm of $u$. Now,
considering  $d = 2$, $r = 4$ and  $p = 2 = q$ in the Proposition  $\ref{GNM}$, we have
\[
\theta = \frac{ 4 ( 4 - 2 )}{4 (2 (2 -2) + 4)} = \frac{2}{(0 + 4)} = \frac{1}{2}.
\]
Thus, it follows from  $(\ref{EQ34})$ that
\begin{equation}\label{EQ35}
\|u\|_{L^{4}}^{4} \leq  \left( A_{opt} \| \nabla u\|_{L^{2}}^{2} + B \|u\|_{L^{2}}^{2} \right) \|u\|_{L^{2}}^{2}.
\end{equation}

 So, denoting $C_{A_{opt}, B} : = \max \{ A_{opt}, B \}$, we obtain
\begin{equation}\label{EQ36}
\|u\|_{L^{4}}^{4} \leq C_{A_{opt}, B}  \|u\|_{H^{1}}^{2} \|u\|_{L^{2}}^{2}.
\end{equation}

Therefore,
\begin{equation}\label{EQ37}
\|u\|_{L^{4}} \leq  C_{A_{opt}, B}^{1/4}  \|u\|_{H^{1}}^{1/2} \|u\|_{L^{2}}^{1/2} .
\end{equation}

Using $(\ref{EQ36})$ in $(\ref{EQ33})$ and the $L^{2}$-norm conservation \eqref{MassConservation}, we obtain
\begin{equation}\label{EQ38}
\begin{split}
\|v(t)\|_{L^{2}} & \leq\|v_{0}\|_{L^{2}} + \frac{1}{\kappa} \left(\int_{0}^{T}   C_{A_{opt}, B}  \|u(t')\|_{H^{1}}^{2} \|u_{0}\|_{L^{2}}^{2}  dt'\right)^{1/2} T^{1/2} \\
& \leq \|v_{0}\|_{L^{2}} + \frac{T}{\kappa} C_{A_{opt}, B}^{1/2}   \Theta_{T}  \|u_{0}\|_{L^{2}},
\end{split}
\end{equation}
where
\begin{equation}\label{ThetaT}
\Theta_{T}:= \displaystyle\sup_{0 \leq t' \leq T} \| u(t')\|_{H^{1}(M)}.
\end{equation}

 On the other hand, from $(\ref{EQ36})$  and $(\ref{ThetaT})$ we get
\begin{equation}\label{EQ39}
\|u(t)\|_{L^{4}}^{2} \leq  C_{A_{opt}, B}^{1/2}  \Theta_{T}  \|u_{0}\|_{L^{2}}.
\end{equation}

Inserting  the estimates  $(\ref{EQ38})$ and  $(\ref{EQ39})$ in $(\ref{EQ32})$, we get
\begin{equation}\label{EQ40}
\begin{split}
I_{1} &\leq \Big( \|v_{0}\|_{L^{2}} + \frac{T}{\kappa} C_{A_{opt}, B}^{1/2} \Theta_{T} \|u_{0}\|_{L^{2}}\Big) \Big(C_{A_{opt}, B}^{1/2} \Theta_{T} \|u_{0}\|_{L^{2}} \Big)\\
& =  C_{A_{opt}, B}^{1/2} \Theta_{T} \|u_{0}\|_{L^{2}} \|v_{0}\|_{L^{2}} +  \frac{T}{\kappa} \Theta_{T}^{2}  C_{A_{opt}, B} \|u_{0}\|_{L^{2}}^{2}.
\end{split}
\end{equation}

\item[$\bullet$]  Now, for the term  $I_{2}$ we use Cauchy-Schwarz inequality, to obtain
\begin{equation}\label{EQ42}
I_{2} = \int_{M} |u_{0}|^{2} |v_{0}| dg \leq \|u_{0}\|_{L^{4}}^{2} \|v_{0}\|_{L^{2}}.
\end{equation}
Using $(\ref{EQ37})$, it follows from $(\ref{EQ42})$ that
\begin{equation}\label{EQ43}
I_{2}  \leq C_{A_{opt}, B}^{1/2} \Theta_{T} \|u_{0}\|_{L^{2}}  \|v_{0}\|_{L^{2}}.
\end{equation}

\item[$\bullet$] To estimate $I_{3}$, we  use the bound  for the term  $I_{1}$ in $(\ref{EQ40})$, and get
\begin{equation}\label{EQ44}
\begin{split}
I_{3} & \leq \int_{0}^{T} \Big( C_{A_{opt}, B}^{1/2} \Theta_{T} \|u_{0}\|_{L^{2}} \|v_{0}\|_{L^{2}} +  \frac{T}{\kappa} \Theta_{T}^{2}  C_{A_{opt}, B} \|u_{0}\|_{L^{2}}^{2} \Big)dt'\\
& \leq  C_{A_{opt}, B}^{1/2} T\Theta_{T} \|u_{0}\|_{L^{2}} \|v_{0}\|_{L^{2}}  +  \frac{T^{2}}{\kappa} \Theta_{T}^{2}  C_{A_{opt}, B} \|u_{0}\|_{L^{2}}^{2}.
\end{split}
\end{equation}

\item[$\bullet$]
To bound the term $I_{4}$, we use $(\ref{EQ36})$, and obtain 
\begin{equation}\label{EQ47}
\begin{split}
I_{4}  &\leq \frac{1}{\kappa} \int_{0}^{T} \|u(t')\|_{L^{4}}^{4} dt' \\
&\leq C_{A_{opt}, B} \frac{1}{\kappa} \int_{0}^{T} \|u(t')\|_{H^{1}}^{2} \|u_{0}\|_{L^{2}}^{2}  dt' \\
 &\leq T \frac{1}{\kappa} C_{A_{opt}, B}  \Theta_{T}^{2} \|u_{0}\|_{L^{2}}^{2} . \\
\end{split}
\end{equation}
\end{enumerate}

Recalling  $(\ref{EQ31})$ and using $(\ref{EQ40})$, $(\ref{EQ43})$, $(\ref{EQ44})$ and $(\ref{EQ47})$, we get
\begin{equation}\label{EQ48}
\begin{split}
\|u(t)\|_{H^{1}}^{2} &\leq \|u_{0}\|_{H^{1}}^{2} + \Big(  C_{A_{opt}, B}^{1/2} \Theta_{T} \|u_{0}\|_{L^{2}} \|v_{0}\|_{L^{2}} +  \frac{T}{\kappa} \Theta_{T}^{2}  C_{A_{opt}, B} \|u_{0}\|_{L^{2}}^{2} \Big)\\
 & \quad+ \Big( C_{A_{opt}, B}^{1/2} \Theta_{T} \|u_{0}\|_{L^{2}}  \|v_{0}\|_{L^{2}}\Big)   \\
 &\quad+ \Big(  C_{A_{opt}, B}^{1/2} T\Theta_{T} \|u_{0}\|_{L^{2}} \|v_{0}\|_{L^{2}}  +   \frac{T^{2}}{\kappa} \Theta_{T}^{2}  C_{A_{opt}, B} \|u_{0}\|_{L^{2}}^{2} \Big) \\
 & \quad+  \Big(  \frac{T}{\kappa} C_{A_{opt}, B}  \Theta_{T}^{2} \|u_{0}\|_{L^{2}}^{2} \Big).
\end{split}
\end{equation}

Thus, using  $T \leq 1$, and further simplifying the estimate $(\ref{EQ48})$, we obtain
\begin{equation}\label{EQ50}
\begin{split}
\|u(t)\|_{H^{1}}^{2} &\leq \|u_{0}\|_{H^{1}}^{2} + 3 C_{A_{opt}, B}^{1/2} \Theta_{T} \|u_{0}\|_{L^{2}} \|v_{0}\|_{L^{2}} +   \frac{T}{\kappa} \Theta_{T}^{2}  C_{A_{opt}, B} \|u_{0}\|_{L^{2}}^{2} \\
& \qquad+   \frac{T^{2}}{\kappa} \Theta_{T}^{2}  C_{A_{opt}, B} \|u_{0}\|_{L^{2}}^{2}  +  \frac{T}{\kappa} C_{A_{opt}, B}  \Theta_{T}^{2} \|u_{0}\|_{L^{2}}^{2}\\
&\leq \|u_{0}\|_{H^{1}}^{2} + 3 C_{A_{opt}, B}^{1/2} \Theta_{T} \|u_{0}\|_{L^{2}} \|v_{0}\|_{L^{2}} +  T \Big(  \frac{3}{\kappa} C_{A_{opt}, B} \|u_{0}\|_{L^{2}}^{2} \Big) \Theta_{T}^{2}.
\end{split}
\end{equation}

Applying the inequality $ab \leq \varepsilon a^{2} + \displaystyle \frac{b^{2}}{4 \varepsilon}$  $(\forall \varepsilon > 0)$, with $\varepsilon = 1/4$
 in the second term in the right hand side of $(\ref{EQ50})$, yields
\begin{equation}\label{EQ51}
 3 C_{A_{opt}, B}^{1/2} \Theta_{T} \|u_{0}\|_{L^{2}} \|v_{0}\|_{L^{2}} \leq \frac{\Theta_{T}^{2}}{4} + \Big( 9 C_{A_{opt}, B}  \|u_{0}\|_{L^{2}}^{2} \|v_{0}\|_{L^{2}}^{2}\Big).
\end{equation}

Thus, if
\[
0 < T \leq \min \Big\{ 1,\frac{\kappa}{  12 C_{A_{opt}, B} \|u_{0}\|_{L^{2}}^{2}} \Big\},
\]
it follows from $(\ref{EQ50})$ and  $(\ref{EQ51})$ that
\begin{equation}\label{EQ52}
\Theta_{T}^{2} \leq \|u_{0}\|_{H^{1}}^{2} + \frac{1}{2} \Theta_{T}^{2} +  9 C_{A_{opt}, B}  \|u_{0}\|_{L^{2}}^{2} \|v_{0}\|_{L^{2}}^{2},
\end{equation}
which gives the desired  estimate $(\ref{PrioriEstimate})$.
\end{proof}

\subsection{Proof of the global well-posedness result}
In this subsection we use the \textit{a priori} estimate $(\ref{PrioriEstimate})$  to prove the global well-posedness result stated in the Theorem $\ref{TeorIntro7}$.

\begin{proof}[Proof of Theorem \ref{TeorIntro7}] By  Theorem $\ref{TeorIntro5}$, we know that for any  $u_{0} \in H^{1}$ and $v_{0} \in H^{1+}$ (which must be seen as being fixed) there exist  $T =  T(\|u_{0}\|_{H^{1}}, \|v_{0}\|_{H^{1^{+}}}, \kappa) > 0$ and  a unique  $H^{1}$-solution  $u \in X_{T}$  of \eqref{IntegroDifferentialFormulation}. Let $ T^{\ast} > 0$ be the maximal time of existence for which

\begin{equation}\label{uMaxT}
u \in C([0,T^{\ast}), H^{1}(M^{2})) \cap L^{p}([0,T^{\ast}), L^{\infty}(M^{2})).
\end{equation}
 We want to show that  $T^{\ast} = + \infty$. Suppose by contradiction that  $T^{\ast} < + \infty$. Let $t_{0} \geq 0$ be fixed, such that
\[
0 < T^{\ast} - t_{0} <  \min \Big\{ 1, \frac{\kappa}{ 12 C_{A_{opt}, B} \|u_{0}\|_{L^{2}}^{2}} \Big\}.
\]

 Choose an increasing sequence $ \{(t_{n})_{n \in \mathbb{N}}\} \subset (t_{0}, T^{\ast})$, with $t_{n} \nearrow T^{\ast}$. Since $t_{n + 1} - t_{0} < T^{\ast} - t_{0}$, we can apply the \textit{a priori} estimate $(\ref{PrioriEstimate})$ over any interval of the form  $ I_{n}:=[t_{0}, t_{n + 1}]$, to obtain
\begin{equation}\label{EQ52}
\|u(t_{n})\|_{H^{1}}^{2} \leq 2 \|u(t_{0})\|_{H^{1}}^{2} +  F(t_0),
\end{equation}
where $ F(t) : = 18 C_{A_{opt}, B} \|u(t)\|_{L^{2}}^{2} \|v(t)\|_{L^{2}}^{2}$.

 Consider the problem
 \begin{equation}\label{EQ53}
\begin{cases}
i \partial_{t} \tilde{u}+ \Delta_{g} \tilde{u} = \tilde{u}  e^{- \frac{t}{\kappa}} v_{0}(x) + \displaystyle \frac{\lambda}{\kappa} \tilde{u} \int_{0}^{t} e^{- \frac{(t - t')}{\kappa}} |\tilde{u}(t')|^{2} d t' , \qquad{ }   \\
\tilde{u}(0) = u(t_{n}) \in H^{1}(M^{2}).\\
\end{cases}
\end{equation}

Recall from \eqref{Choice-T2} that the function $[0, \infty) \rightarrow (0, \infty)  $, $ R \mapsto T(R, \|v_{0}\|_{H^{1+}}, \kappa)$, where $R$ depends on $\|u_{0}\|_{H^{1}}$ is a decreasing function and gives the time of existence of a solution $u$. In the case of the time existence to \eqref{EQ53}, 
we consider $R_{n}: = \|u(t_{n})\|_{H^{1}}$. From \eqref{EQ52}, we obtain $R_{n} \leq R_{0}$ for all $n \in \mathbb{N}$, where
\[
 R_{0} : = \sqrt{2 \|u(t_{0})\|_{H^{1}}^{2} +  F(t_0)} . 
\]

Thus, if  $T_{n} :=T(R_{n}, \|v_{0}\|_{H^{1+}}, \kappa)$ denotes the time of local existence for the solution 
of \eqref{EQ53}, (which is given by Theorem \ref{TeorIntro5}) we have
\begin{equation}\label{Tn}
0 < c_{0}:=T(R_{0}, \|v_{0}\|_{H^{1+}}, \kappa)\leq T_{n} \qquad \mbox{for all} \qquad  n \in \mathbb{N},
\end{equation}
 with $\tilde{u} \in C([0, T_{n}], H^{1}(M)) \cap L^{p}([0, T_{n}], L^{\infty}(M))$ being the unique solution to $(\ref{EQ53})$. 
 Choose  $n_{0}$  in such a way  that  $ T^{\ast} - t_{n_{0}} <  c_{0}$ and consider the functions defined by
\begin{equation}\label{uThec}
\breve{u}(t) := \begin{cases}
  u(t)  \quad\qquad{ }& \mbox{ if } \qquad 0 \leq t \leq t_{n_{0}},  \\
  \tilde{u}(t - t_{n_{0}}) &\mbox{ if } \qquad t_{n_{0}} \leq t \leq t_{n_{0}} + T_{n_{0}},\\
\end{cases}
\end{equation}
and
\begin{equation}\label{vThec}
\breve{v}(t) := \begin{cases}
  v(t) \quad \qquad &\mbox{ if } \qquad 0 \leq t \leq t_{n_{0}},  \\
  \tilde{v}(t - t_{n_{0}})&  \mbox{ if } \qquad t_{n_{0}} \leq t \leq t_{n_{0}} + T_{n_{0}},\\
\end{cases}
\end{equation}
where, for  $0 \leq \ell \leq T_{n_{0}}$
\begin{equation}\label{Tildev}
\tilde{v}(\ell) = e^{- \frac{\ell}{\kappa}} \tilde{v}(0) + \int_{0}^{\ell} e^{-\frac{(\ell - t')}{\kappa}} |\tilde{u}(t')|^{2} dt'.
\end{equation}
The function $\tilde{v}$ defined here is  not to be confused with the nonlinear part of the auxiliar problem \eqref{EQ53}.

 By compatibility, we must have $\tilde{v}(0) = v(t_{n_{0}})$.  Using definitions \eqref{uThec} and \eqref{vThec}, it is easy to see that $\breve{u}$ satisfies  the integral equation
\[
\breve{u}(t) = S(t)u_{0} -  i \int_{0}^{t} S(t - \ell) \breve{u}(\ell) \breve{v}(\ell) d \ell,
\]
  for  $0 \leq t \leq t_{n_{0}} + T_{n_{0}}$.  
  
 First, we show that 
 $
\breve{u} \in L^{p}([0,t_{n_{0}} + T_{n_{0}}],L^{\infty}(M))$. Using the definitions of $u$ and  $\breve{u}$, we can write
\[
\|\breve{u}\|_{L^{p}([0, t_{n_{0}} + T_{n_{0}}], L^{\infty}(M))}^{p} = \int_{0}^{t_{n_{0}}} \|u(t')\|_{L^{\infty}(M)}^{p} dt' +  \int_{t_{n_{0}}}^{t_{n_{0}} + T_{n_{0}}} \|\breve{u}(t')\|_{L^{\infty}(M)}^{p} dt'.
\]

 Since $u \in L^{p}([0,t_{n_{0}}],L^{\infty}(M))$, we have  $ \int_{0}^{t_{n_{0}}} \|u(t')\|_{L^{\infty}(M)}^{p} dt'< \infty$. Performing a change of variables  $ \ell = t' - t_{n_{0}}$, we get
\[
\int_{t_{n_{0}}}^{t_{n_{0}} + T_{n_{0}}} \|\breve{u}(t')\|_{L^{\infty}(M)}^{p} dt' = \int_{t_{n_{0}}}^{t_{n_{0}} + T_{n_{0}}} \|\tilde{u}(t' - t_{n_{0}})\|_{L^{\infty}(M)}^{p} dt' = \int_{0}^{ T_{n_{0}}} \|\tilde{u}(\ell)\|_{L^{\infty}(M)}^{p} d\ell < \infty,
\]
as required.

Now, to show that  $\breve{u} \in C([0, t_{n_{0}} + T_{n_{0}}], H^{1}(M)) $, it suffices to verify continuity at
  $t_{n_{0}}$, which is the ``gluing point". To show continuity from the right, we prove
\[
 \|\breve{u}(t_{n_{0}} + \varepsilon) - \breve{u}(t_{n_{0}})\|_{H^{1}} \longrightarrow 0,
\]
when $\varepsilon \longrightarrow 0^{+}$. Note that
\begin{equation}\label{EQ54}
\breve{u}(t_{n_{0}} + \varepsilon) = \tilde{u}(\varepsilon) = S(\varepsilon) \tilde{u}(0) -  i \int_{0}^{\varepsilon} S(t_{n_{0}}- \ell) \tilde{u}(\ell)  \tilde{v}(\ell) d \ell,
\end{equation}
where
\begin{equation}\label{EQ55}
\tilde{u}(0) = u(t_{n_{0}}) = S(t_{n_{0}}) u_{0} -  i \int_{0}^{t_{n_{0}}} S(t_{n_{0}}- \ell) u(\ell)v(\ell) d \ell.
\end{equation}

Thus, it follows from  $(\ref{EQ54})$ and $(\ref{EQ55})$ that
\begin{equation}\label{EQ56}
\breve{u}(t_{n_{0}} + \varepsilon) =  S(t_{n_{0}} + \varepsilon) u_{0} -  i \int_{0}^{t_{n_{0}}} S(t_{n_{0}} + \varepsilon- \ell) u(\ell)v(\ell) d \ell
-  i \int_{0}^{\varepsilon} S(t_{n_{0}}- \ell) \tilde{u}(\ell)  \tilde{v}(\ell) d \ell.
\end{equation}

On the other hand, using  \eqref{uThec} we have
\begin{equation}\label{EQ57}
\breve{u}(t_{n_{0}}) = u(t_{n_{0}}) = S(t_{n_{0}}) u_{0} - i  \int_{0}^{t_{n_{0}}} S(t_{n_{0}} - \ell) u(\ell) v(\ell) d \ell.
\end{equation}

Subtracting $(\ref{EQ56}) - (\ref{EQ57})$ in the  $H^{1}$ norm, and using that $S(t)$ is an isometry in  $H^{1}$, we obtain
\begin{equation}\label{EQ58}
\begin{split}
\|\breve{u}(t_{n_{0}} + \varepsilon) - \breve{u}(t_{n_{0}})\|_{H^{1}} &\leq \|S(\varepsilon) u_{0} -  u_{0}\|_{H^{1}} +  \int_{0}^{t_{n_{0}}} \|S(\varepsilon)  u(\ell) v(\ell) - u(\ell) v(\ell) \|_{H^{1}} d \ell + \int_{0}^{\varepsilon}\| \tilde{u}(\ell)  \tilde{v}(\ell)\|_{H^{1}} d \ell \\
& =: A_{1}(\varepsilon) + A_{2}(\varepsilon) + A_{3}(\varepsilon). \\
\end{split}
\end{equation}

As $\{ S(t) \}_{t \in \mathbb{R}}$ is a continuous semigroup,
\begin{equation}\label{EQ59}
A_{1}(\varepsilon) \longrightarrow 0 \mbox{ } \mbox{ when } \mbox{ } \varepsilon \longrightarrow 0^{+}.
\end{equation}

To prove the convergence of the  term $A_{2}(\varepsilon)$ we  use  the dominated convergence theorem. First, using the semigroup property, one has
\begin{equation}\label{EQ60}
\|S(\varepsilon)  u(\ell) v(\ell) - u(\ell) v(\ell) \|_{H^{1}}  \longrightarrow 0 \mbox{ } \mbox{ when } \mbox{ } \varepsilon \longrightarrow 0^{+},
\end{equation}
pointwise, for  $\ell \in (0, t_{n_{0}})$. Now, observe that
\begin{equation}\label{EQ61}
\| u(\ell)  v(\ell)\|_{H^{1}} \leq C \| u\|_{ L^{\infty}([0, t_{n_{0}}],H^{1})} \|v(\ell)\|_{L^{\infty}} + C \|u(\ell)\|_{L^{\infty}} \|v\|_{ L^{\infty}([0, t_{n_{0}}],H^{1})}.
\end{equation}

Thus, 
\begin{equation}\label{EQ62}
\int_{0}^{t_{n_{0}}}\| u(\ell)  v(\ell)\|_{H^{1}} d \ell  \leq   C \| u\|_{ L^{\infty}([0, t_{n_{0}}],H^{1})} \|v\|_{ L^{1}([0, t_{n_{0}}],L^{\infty})} + C \|u\|_{ L^{1}([0, t_{n_{0}}],L^{\infty})} \|v\|_{ L^{\infty}([0, t_{n_{0}}],H^{1})}.
\end{equation}
Using \eqref{uMaxT} and Corollary \ref{v(t)Persistence} (persistence property for $v$), we conclude that the RHS of \eqref{EQ62} is finite. Thus, 
the function $\ell \mapsto \| u(\ell)  v(\ell)\|_{H^{1}}$ for $\ell \in (0,t_{n_{0}})$ is integrable and  
 applying the dominated convergence theorem, we conclude that 
$\lim_{\varepsilon \rightarrow 0^{+}} A_{2}(\varepsilon) = 0$ .

Now, notice that 
\[
A_{3}(\varepsilon) =  \int_{0}^{T_{n_{0}}} \chi_{[0, \varepsilon]}(\ell) \|\tilde{u}(\ell)  \tilde{v}(\ell)\|_{H^{1}} d \ell.
\]
To use the reasoning as in the term $A_{2}$, first note that $\tilde{u} \in X_{T_{n_{0}}} := C([0,T_{n_{0}}], H^{1}(M)) \cap L^{p}([0, T_{n_{0}}],L^{\infty}(M))$.
As a consequence of persistence of $\tilde{v}$ we also have $\tilde{v} \in X_{T_{n_{0}}}$.
Therefore, using the dominated convergence theorem, we conclude that 
$\lim_{\varepsilon \rightarrow 0^{+}}  A_{3}(\varepsilon) = 0$.

Next, to show continuity from the left, we must show
\[
 \| \breve{u}(t_{n_{0}} - \varepsilon) - \breve{u}(t_{n_{0}})\|_{H^{1}} \longrightarrow 0,
\]
when $\varepsilon \longrightarrow 0^{+}$. By the definition of $\breve{u}$, we have:
\begin{equation}\label{EQ67}
\breve{u}(t_{n_{0}}) = S(t_{n_{0}}) u_{0} -  i \int_{0}^{t_{n_{0}}} S(t_{n_{0}}-\ell) \tilde{u}(\ell)  \tilde{v}(\ell) d \ell.
\end{equation}
and
\begin{equation}\label{EQ68}
\breve{u}(t_{n_{0}} -\varepsilon) = S(t_{n_{0}} -\varepsilon) u_{0} -  i \int_{0}^{t_{n_{0}} -\varepsilon} S(t_{n_{0}}- \varepsilon -\ell) \tilde{u}(\ell)  \tilde{v}(\ell) d \ell.
\end{equation}
Thus, performing the subtraction  (\ref{EQ67}) $-$ (\ref{EQ68}), one has
\begin{equation}\label{EQ69}
\begin{split}
 \| \breve{u}(t_{n_{0}}) - \breve{u}(t_{n_{0}} - \varepsilon) \|_{H^{1}} &\leq  \|S(t_{n_{0}} -\varepsilon) u_{0} - S(t_{n_{0}}) u_{0}\|_{H^{1}} +   \\
 & \qquad+ \Big\| \int_{0}^{t_{n_{0}} - \varepsilon} S(t_{n_{0}} - \varepsilon - \ell) u(\ell) v(\ell) d \ell -  \int_{0}^{t_{n_{0}}} S(t_{n_{0}}  - \ell) u(\ell) v(\ell) d \ell  \Big\|_{H^{1}}\\
& =:B_{1}(\varepsilon) + B_{2}(\varepsilon). \\
\end{split}
\end{equation}

Clearly, $\lim_{\varepsilon \rightarrow 0^{+}} B_{1}(\varepsilon) = 0$. Now, for the term $B_{2}(\varepsilon)$, we have
\begin{equation}
\begin{split}
B_{2}(\varepsilon) &\leq   \Big\| \int_{0}^{t_{n_{0}}} S(t_{n_{0}} - \varepsilon - \ell) u(\ell) v(\ell) d \ell  - \int_{0}^{t_{n_{0}} - \varepsilon} S(t_{n_{0}} - \varepsilon - \ell) u(\ell) v(\ell) d \ell  \Big\|_{H^{1}}  \\
 & \qquad+  \Big\| \int_{0}^{t_{n_{0}}} S(t_{n_{0}} - \varepsilon - \ell) u(\ell) v(\ell) - S(t_{n_{0}}  - \ell) u(\ell) v(\ell) d \ell  \Big\|_{H^{1}}. \\
\end{split}
\end{equation}

Thus,
\begin{equation}\label{EQ70}
 B_{2}(\varepsilon)\leq  \int_{t_{n_{0}} - \varepsilon}^{t_{n_{0}}} \Big\| u(\ell) v(\ell)  \Big\|_{H^{1}} d \ell
 + \int_{0}^{t_{n_{0}}} \Big\|  S( - \varepsilon ) u(\ell) v(\ell) -  u(\ell) v(\ell) \Big\|_{H^{1}} d \ell.
\end{equation}

Using the dominated convergence theorem, we conclude that $\lim_{\varepsilon\rightarrow 0^{+}}  B_{2}(\varepsilon) = 0 $. Hence,
\[
\lim_{\varepsilon\rightarrow 0^{+}} \| \breve{u}(t_{n_{0}}) - \breve{u}(t_{n_{0}} - \varepsilon) \|_{H^{1}} = 0,
\]
 and consequently
\[
 \breve{u} \in C([0, t_{n_{0}} + T_{n_{0}}], H^{1}(M)) \cap L^{p}([0, t_{n_{0}} + T_{n_{0}}], L^{\infty}(M)).
\]
 This shows that, $\breve{u}$ is a unique solution to the  integro-differential  equation in \eqref{IntegroDifferentialFormulation} on the interval $[0, t_{n_{0}} + T_{n_{0}}]$  which contradicts the maximality of $T^{\ast}$ because in view of our choice of $n_{0}$,
\[ 
 T^{\ast} < t_{n_{0}} + c_{0}
\]  
 and by \eqref{Tn}, we have
\[ 
 T^{\ast} < c_{0} + t_{n_{0}} \leq T_{n_{0}} + t_{n_{0}} .
\]  
 So, we must have that $T^{\ast} = +\infty$, i.e, $u$ is a global solution.
\end{proof}


\begin{obs}
In dimension $2$, let us compare our global well-posedness result in $H^{1}(M)$  for  the IVP  $(\ref{SDE12})$ with that of the 
cubic NLS equation \eqref{NLS} in the focusing case, i. e.,
\begin{equation}\label{3FocNLS}
i \partial_{t} u + \Delta u = - |u|^{2} u ; \qquad u(0, x) = u_{0}(x). 
\end{equation}
Recall that, in the defocusing  case using energy conservation law,  the $H^1$--local solution to the cubic NLS equation can be extended globally in time, i.e.,  $u \in C([0, T];H^{1}(M))$ for any bounded $T>0$ (see \cite{BGT}, pg. 571). Nevertheless, in the focusing case, the solution $u(t, \cdot)$ of equation \eqref{3FocNLS} has
the $H^1$-- conserved quantity 
\begin{equation}\label{HamDim2}
E(u(t)) := \frac{1}{2}\|\nabla u(t)\|_{L^{2}}^{2} - \frac{1}{4}\| u(t)\|_{L^{4}}^{4} = E(u_{0}),
\end{equation}
which shows that  $E(u)$ does
not control the $H^{1}$--norm of the solution. In order to overcome this problem, 
note that from \eqref{HamDim2}, one has
\begin{equation}\label{EQ76}
\|\nabla u(t)\|_{L^{2}}^{2} =  \frac{1}{2}\| u(t)\|_{L^{4}}^{4} + 2 E(u_{0}). 
\end{equation}
Now, applying \eqref{EQ35} (which is a consequence of the sharp Gagliardo-Nirenberg inequality \eqref{EQ34}) in \eqref{EQ76}, it follows  that
\begin{equation}\label{EQ77}
\|\nabla u(t)\|_{L^{2}}^{2} \leq A_{0} \|\nabla u(t)\|_{L^{2}}^{2} + B_{0}, 
\end{equation}
where $A_{0} := \frac{A_{opt}}{2} \|u_{0}\|_{L^{2}}^{2}$ and $B_{0} : = 2 B \|u_{0}\|_{L^{2}}^{2} E(u_{0}) $. If we  impose a smallness assumption on the $L^{2}$-norm, namely
\[
A_{0} < 1 \Longleftrightarrow  \|u_{0}\|_{L^{2}} <  \sqrt{\frac{2}{A_{opt}}}, 
\] 
 we can get an $H^1$ {\em a priori} estimate from \eqref{EQ77} and consequently  global well-posedness jn $H^{1}$ for the NLS equation in the focusing case. For a detailed information of this sort of phenomena in $\R^d$ we refer to Theorem 6.2 in \cite{LP-15}.
 
 The discussion above  reveals
a novel phenomena regarding  the global well-posedness of the IVP \eqref{SDE12} in the focusing case.  
The  structure of the nonlinear term of the integro-differential equation  $(\ref{IntegroDifferentialFormulation})$ allowed us
 to obtain the same global well-posedness result for the SD system in the both defocusing and focusing cases without any
smallness hypothesis on the data. This subtle difference between the NLS equation  and the SD system $(\ref{SDE12})$
occurs because the (coupled) evolution of $v(t, \cdot)$ in  \eqref{IntegroDifferentialFormulation} permits to derive an {\em a priori}
estimate for the $H^{1}$ norm of $u$, although we do not have other known conserved quantities (other than the $L^{2}$-norm).
\end{obs}

\subsection{Growth estimate of the $H^{1}$-solution in the case $\lambda = 1$ and  $v_{0} \geq 0$.}\label{sub-sec4.3}
Suppose $\lambda = 1$ , $v_{0} \geq 0$ and $\kappa =1 $ (to simplify). So, $v(t)\geq 0$ and in this case,
 the estimate   $(\ref{EQ30})$ is reduced to
\begin{equation}\label{EQ71}
\|u(t)\|_{H^{1}}^{2}  \leq \|u_{0}\|_{H^{1}}^{2} + \int_{M} |u_{0}|^{2} v_{0} + \int_{0}^{t} \|u(t')\|_{L^{4}}^{4} d t'.
\end{equation}

Applying the estimate $(\ref{EQ36})$  we obtain
\begin{equation}\label{EQ72}
\|u(t)\|_{H^{1}}^{2}  \leq \|u_{0}\|_{H^{1}}^{2} +\|u_{0}\|_{L^{4}}^{2} \|v_{0}\|_{L^{2}} + \int_{0}^{t}  C_{A_{opt}, B} \|u_{0}\|_{L^{2}}^{2}\|u(t')\|_{H^{1}}^{2} d t'.
\end{equation}
Also using $(\ref{EQ36})$, followed by Young's inequality, we have
\begin{equation}\label{EQ73}
\|u_{0}\|_{L^{4}}^{2} \|v_{0}\|_{L^{2}}   \leq \frac{ C_{A_{opt}, B} \|u_{0}\|_{L^{2}}^{2}}{2} \|v_{0}\|_{L^{2}} + \frac{\|u_{0}\|_{H^{1}}^{2}}{2}  \|v_{0}\|_{L^{2}}.
\end{equation}

Substituting \eqref{EQ73} in \eqref{EQ72}, we get 
\begin{equation}\label{EQ74}
\|u(t)\|_{H^{1}}^{2}  \leq C_{0} + \int_{0}^{t} A_{0}\|u(t')\|_{H^{1}}^{2} d t',
\end{equation}
where
$$
A_{0} : = C_{A_{opt}, B} \|u_{0}\|_{L^{2}}^{2},\; {\mathrm{and}}\; C_{0} :=\|u_{0}\|_{H^{1}}^{2} +  \frac{A_{0}}{2}\|v_{0}\|_{L^{2}}+ \frac{\|u_{0}\|_{H^{1}}^{2}}{2}  \|v_{0}\|_{L^{2}}.
$$

Applying the Gronwall's inequality in $(\ref{EQ74})$ we obtain that for  $0 \leq t < T^{\ast}$
\begin{equation}\label{EQ75}
\|u(t)\|_{H^{1}}^{2}  \leq C_{0} \cdot e^{\int_{0}^{t} A_{0} dt'} = C_{0} \cdot e^{A_{0} t}.
\end{equation}
This shows that the growth of the $H^{1}$ -solution in the case  $\lambda = 1$ and  $v_{0} \geq 0$ is at  most exponential.

\section{Concluding remarks}\label{sec-5}
The additional $L^{\infty}$ condition used on the initial data $v_0$ to obtain local well-posedness result is of technical character. If   working on $\R^d$, one  can take advantage of  the Strichartz estimate without loss. More explicitly, 
for the solutions $u$ of the nonlinear Schrödinger equation
\[
i \partial_{t} u + \Delta_{\mathbb{R}^{d}} u = F(u); \qquad u(0, x) = u_{0},
\]
 one has the following   estimate in the Duhamel's formula
\[
\|u\|_{L^{p_{1}}_{T} L_{x}^{q_{1}}} \lesssim \|u_0\|_{L^{2}} + \|F(u)\|_{L^{p_{2}'}_{T} L_{x}^{q_{2}'}}
\]
where $(p_{1}, q_{1})$ and $(p_{2}, q_{2})$ are admissible pairs,
$1/p_{2} + 1/p_{2}' = 1$ and $1/q_{1} + 1/q_{2}' = 1$.  In this framework, if we work  with the nonlinear term $G$ such as in \eqref{Wintegral}
rather than $F$, the structure of the  mixed Lebesgue norm allows to deal better with the nonlinear term $e^{-t/\kappa }u(t,x) v_{0}(x)$ which leads to remove 
the extra condition on the initial data $v_{0}$.

 If  $M = \T^d$  one may take advantage of the expansion in the Fourier series. In fact  $u \in C_0^{\infty}(\mathbb{R} \times \mathbb{T}^{d})$ can be written as 
 \[
u(t, x) = \sum_{m \in \mathbb{Z}^{d}} e^{2 \pi i x \cdot m} \int_{\mathbb{R}} \widehat{u}(m, \tau) e^{ 2 \pi i t \tau} d \tau,
\] 
where the symbol \, $\widehat{ \cdot}$  \, means  the Fourier transform taken in both  space and time variables.  This property was nicely explored in \cite{BOURG93A} to obtain crucial estimates in the NLS case.

When $(M, g)$ is a general compact Riemannian manifold the situation is the following. As noted in \cite{BGT}, on  $M$ a
natural generalization of Fourier series expansions is of course the spectral decomposition
of $-\Delta_g$. More explicitly, if  
 $u \in C_0^{\infty}(\mathbb{R} \times M^{d})$ we have an analogue of the Fourier expansion  given by 
\[
u(t, x) = \sum_{\lambda \in Spec(-\Delta_{g})} e_{\lambda}(x) \int_{\mathbb{R}} \widehat{P_{\lambda} u}(y, \tau) e^{2 \pi i t \tau} d\tau,
\]
where $e_{\lambda}$ are the eigenfunctions of $- \Delta_{g}$ given by the relation $- \Delta_{g} e_{\lambda} = \lambda e_{\lambda} $ and $P_{\lambda}$ is the spectral projection onto the eigenspace related to  the eigenvalue $\lambda$. Therefore, the space-time Fourier transform  $\widehat{u}(m, \tau)$, in the general context, is replaced by   $P_{\lambda} \widehat{u}(\cdot, \tau)$ which behaves badly in the convolution product.

 Besides, the spectrum and  eigenfunctions of the Laplace-Beltrami operator on arbitrary Riemannian manifolds are not well known in order  to adapt Bourgain's method. However, on compact surfaces like Zoll manifolds one has better knowledge of spectrum and estimates involving eigenfunctions  which is crucial to adapt Bourgain's method. These properties are used very nicely in the case of the NLS equation (see for example \cite{JIANG2011,HANI2012,BGT2004} and references therein). Motivated by these works, we are  adapting Bourgain's space framework to address  the well-posedness issues for the IVP associated to the SD system \eqref{SDE12} posed on compact surfaces like Zoll manifolds. Work in this direction  is in progress.


\subsection*{Acknowledgment}
The authors would like to thank Prof. Adan Corcho from Federal University of Rio de Janeiro and Prof. Ademir Pastor from University of Campinas for many helpful conversations.


\bibliographystyle{plain}

\end{document}